\documentclass[12pt,reqno]{amsart}

\usepackage[leqno]{amsmath}
\usepackage{amssymb,amsfonts,xfrac,MnSymbol}
\usepackage{amsthm}
\usepackage{cite}
\usepackage{hyperref}
\usepackage{todonotes}
\usepackage{enumitem}
\usepackage{graphicx}
\usepackage{tikz-cd, tikz}
\usepackage{color}
\usepackage{import}
\usepackage[all]{xy}
\usepackage[margin=1.2in]{geometry}
\usepackage[perc]{overpic}
\usepackage{subfigure}
\usepackage[all]{xy}
\usepackage{caption}
\usepackage{xspace}
\usepackage{float}

\numberwithin{equation}{section}

\makeatletter
\def\thm@space@setup{%
  \thm@preskip=12pt plus 4pt minus 6pt
  \thm@postskip=15pt
}
\makeatother

\newtheorem{theorem}{Theorem}[section]
\newtheorem*{theorem*}{Theorem}
\newtheorem{proposition}[theorem]{Proposition}
\newtheorem*{proposition*}{Proposition}
\newtheorem{lemma}[theorem]{Lemma}
\newtheorem*{lemma*}{Lemma}
\newtheorem{claim}[theorem]{Claim}
\newtheorem*{claim*}{Claim}

\newtheorem*{subclaim*}{Sub-claim}
\newtheorem{corollary}[theorem]{Corollary}
\newtheorem*{corollary*}{Corollary}
\newtheorem{theoremA}{Theorem}

\theoremstyle{definition}
\newtheorem{definition}[theorem]{Definition}

\newtheorem{remark}[theorem]{Remark}
\newtheorem{example}[theorem]{Example}

\newtheorem*{definition*}{Definition}
\newtheorem*{observation*}{Observation}
\newtheorem*{remark*}{Remark}
\newtheorem*{example*}{Example}
\newtheorem*{question*}{Question}
\newtheorem*{exercise*}{Exercise}
\newtheorem*{fact*}{Fact}
\newtheorem*{notation*}{Notation}

\newtheorem*{conjecture*}{Corollary}

\newcommand{\bbN}{\mathbb{N}}

\newcommand{\bbS}{\mathbb{S}}
\newcommand{\bbT}{\mathbb{T}}

\newcommand{\bfT}{\mathbf{T}}

\newcommand{\calB}{\mathcal{B}}
\newcommand{\calC}{\mathcal{C}}
\newcommand{\calD}{\mathcal{D}}

\newcommand{\calK}{\mathcal{K}}
\newcommand{\calL}{\mathcal{L}}

\newcommand{\calN}{\mathcal{N}}
\newcommand{\NN}{\mathcal{N}}

\newcommand{\calT}{\mathcal{T}}

\newcommand{\ii}{^{-1}}

\newcommand{\ssm}{\smallsetminus}
\newcommand{\abs}[1]{\left| #1 \right|}

\newcommand{\bbl}[1]{\mathfrak{J}_\circ(#1)}
\newcommand{\bdr}[1]{\mathfrak{J}_\partial(#1)}
\newcommand{\bbb}[1]{\mathfrak{J}(#1)}
\newcommand{\pos}{>0}
\newcommand{\np}{\le 0}

\newcommand{\I}{$\circ$}
\newcommand{\D}{$\partial$}

\newcommand{\joint}{joint\xspace}

\title{Highly twisted diagrams}
\author{Nir Lazarovich}
\author{Yoav Moriah}

\author{Tali Pinsky}
\thanks{NL was supported by the Israel Science Foundation (grant no. 1562/19).}

\date{}



\address{Department of Mathematics\\
Technion\\
Haifa, 32000 Israel}
\email{lazarovich@technion.ac.il}

\email{ymoriah@technion.ac.il}

\email{talipi@technion.ac.il}

\begin{document}

\maketitle

\begin{abstract} We prove that the knots and links that admit a $3$-highly 
twisted irreducible diagram  with more than two twist regions are hyperbolic. 
This should be compared with  a result of Futer-Purcell for $6$-highly 
twisted diagrams. While their proof uses  geometric methods our  proof is 
achieved by  showing that the  complements of such  knots or links are 
unannular and atoroidal. This is  done by using a new approach 
involving an Euler  characteristic argument. 

\end{abstract}

\section{introduction}\label{sec: introduction}

The prevailing feeling among low dimensional topologists is that ``most" links 
$\mathcal{L}$ in  $\bbS^3$ are hyperbolic. That means that the open manifold 
$\bbS^3 \ssm \mathcal{L}$ can be endowed with a complete, finite-volume, hyperbolic 
metric of sectional curvature $-1$. Being hyperbolic is a property of the manifold 
with far reaching consequences. However, proving that a specific link 
$\mathcal{L}$ is hyperbolic turns out to be non trivial. This is especially 
true if  the link $\mathcal{L}$ is ``heavy duty", i.e., has a very large 
crossing number. See for example
\cite{minsky-moriah:surplus}.

The question of when can one decide if the complement of a  link in $\bbS^3$ is a 
hyperbolic manifold  from a projection diagram  has been of interest for a
long time. The first result  in this direction 
is by Hatcher and Thurston who proved that  complements of $2$-bridge knots which 
have at least two twist regions (they are not torus knots or links) are  hyperbolic, 
see \cite{Hatcher-Thurston}. The second is Menasco's result \cite{Menasco} that
a non-split prime alternating link which is not a torus link is hyperbolic.  Later  
Futer  and Purcell proved in  \cite{FuterPurcell}, among other results, that every 
link with a   $6$-highly twisted, twist-reduced diagram which has at least  two 
twist regions  is hyperbolic. Their result is  obtained by applying Marc Lackenby's 
$6$-surgery theorem, see \cite{lackenby2000word}, to the corresponding fully augmented 
links. Our main theorem is the following (for definitions see \S\ref{sec: preliminaries}):

\begin{theoremA}\label{thm: highly twisted implies hyperbolic}
Let $D(\calL)$ be a connected, prime, twist-reduced, $3$-highly twisted diagram of a
link $\calL$ with at least two twist regions, then $\calL$ is hyperbolic.
\end{theoremA}

 This result is sharp since  there are non-hyperbolic links with 2-highly twisted 
 link diagrams, as Figure \ref{fig: non hyperbolic counterexample} shows.

\vskip7pt

Theorem \ref{thm: highly twisted implies hyperbolic} weakens the conditions 
imposed in \cite{FuterPurcell} on $D(\mathcal{L})$ from $6$-highly twisted 
to $3$-highly twisted. Two other relevant results in the spirit of  
Theorem \ref{thm: highly twisted implies hyperbolic} are \cite{giambrone} 
and \cite{futerPurcellhyperbolic}. They replace the condition that a diagram 
be 6-highly twisted by conditions related to its ``semi-adequacy'' (as defined 
there). Theorem \ref{thm: highly twisted implies hyperbolic} is a generalization 
of previous work of the authors \cite{lazarovich2021highly}, where a similar 
result was proved for knots/links that have 3-highly twisted plat diagrams.

Clearly not all links have a  diagram that satisfies the conditions of 
Theorem \ref{thm: highly twisted implies hyperbolic}. 
However the  subset of links that do is a ``large'' subset in a sense that can 
be made precise, see the discussion in \cite{lustig2012large}. 

\begin{figure}
    \centering
    \includegraphics{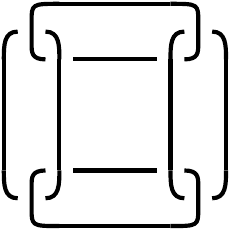}
    \caption{A non-hyperbolic link with a 2-highly twisted diagram.}
    \label{fig: non hyperbolic counterexample}
\end{figure}

As a corollary to Theorem~\ref{thm: highly twisted implies hyperbolic} we obtain a 
simple method to construct essential surfaces in complements of links with highly 
twisted diagrams. This is stated in Section \ref{sec: essential holed spheres} 
as Theorem \ref {thm: General vertical spheres}.

\vskip5pt

Although there are non-hyperbolic links with 2-highly twisted diagrams, we believe 
that the conditions of Theorem \ref{thm: highly twisted implies hyperbolic} could be 
weakened so as to include all alternating links (cf. \cite{Menasco}) and allow \emph{some} 
twist regions to have less than three crossings.
 
\vskip15pt

\section{preliminaries}\label{sec: preliminaries}

\subsection{Bubbles and twist regions.}\label{subsec: bubbles and twist regions}

\vskip7pt

Let $\calL \subset \bbS^3$ be a link. The projection of a link $L$ in the isotopy class 
$\mathcal{L}$ onto a plane $P$ together with the crossing data is a \emph{link diagram} of 
$L$ and is denoted by $D(L)$.  Let $\varepsilon$ be sufficiently small so that the closed
$\varepsilon$-balls around the crossings of $D(L)$ are disjoint. Let $\calB_1,\dots,\calB_r$
be $\varepsilon$-balls around the $r$ crossings of the diagram. The boundaries 
$B_i = \partial \calB_i$, $1\le i \le r$, are the \emph{bubbles} of the diagram. 
The link $\mathcal{L}$ is isotopic to a link $L$ 
which is embedded in $P \cup \bigcup_i B_i$. 
Note that 
$P$ divides each bubble into two hemispheres denoted  by $B^+_i$ and $B^-_i$.
Denote the two 2-spheres $$P^\pm  = (P \ssm \bigcup_i \calB_i) \cup \bigcup_i B^\pm _i.$$  
Each of $P^+,P^-$ bounds  a $3$-ball $H^\pm$ in $\bbS^3\ssm L$.

A \emph{twist region} $T$ is a disk in $P$ which contains a maximal (with respect 
to inclusion) chain of bigons in $D(L)$ describing a trivial integer 2-tangle. 
See Figure \ref{fig:Twist regions} for an example of a diagram with twist regions.
We will assume that a twist region contains the projection of the bubbles around 
the crossings in $T$. We will often abuse terminology, and use twist regions to 
refer to the regions in $P^\pm$ which project to twist regions. Correspondingly, 
we treat the bubbles around the crossings of $T$ as being contained in $T$. 
A \emph{twist box} is the tangle $(\bfT,t)$ where $\bfT$ is the product 
$T\times [-2\varepsilon,2\varepsilon]$ where $T$ is a twist region, and $t$ is 
the tangle  $\bfT \cap L$ The disk diagram $D(L)$ uniquely decomposes into 
disjoint twist regions. 

\begin{figure}
\includegraphics[width = 0.3\textwidth]{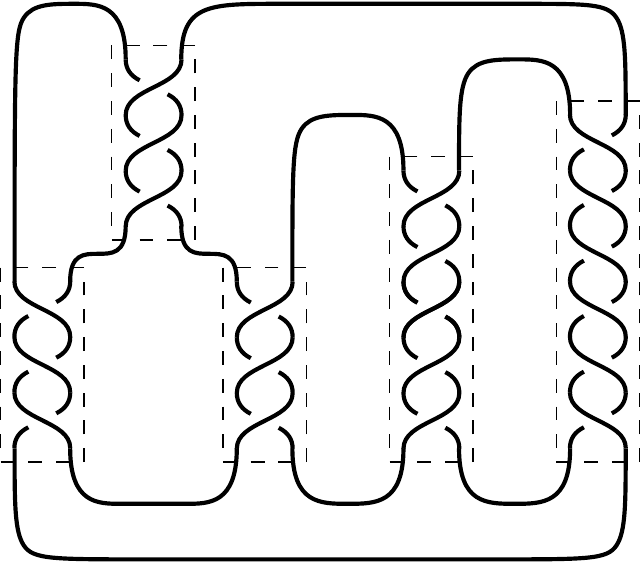}
\caption{A 3-highly twisted link diagram. The twist regions are the dashed rectangles.}
\label{fig:Twist regions}
\end{figure}

\begin{definition}\label{def: twist reduced}
Let $D(L)$ be a link diagram. 
\begin{enumerate}
    \item The diagram $D(L)$ is \emph{prime} if any simple closed curve in $P$ intersecting  
    $D(L)$ transversely in two points bounds a subdiagram with no crossings.
    \item A \emph{twist-reduction subdiagram} is a subdiagram of $D(L)$ enclosed by a 
    simple closed curve $\gamma$ in $P$ which intersects the edges of $\gamma$ transversely 
    in four points composed of two pairs each of which is adjacent to a crossing of $D(L)$ 
    but which is not a chain of bigons describing an integer 2-tangle. The diagram  $D(L)$ 
    is  \emph{twist-reduced} if it contains no twist-reduction subdiagram.
    \item For $k\in\bbN$, the diagram is \emph{$k$-highly twisted} if every twist region 
    has at least $k$ crossings.
\end{enumerate}
\end{definition}

Note that every diagram can be made twist-reduced by performing flypes on 
twist-reduction subdiagrams.

\begin{definition}
A twist region $T$ intersects the link diagram in four points, dividing its boundary $\partial T$
into four segments. If the twist region has at least two crossings, then a pair of opposite 
segments of $\partial T$ can be called the \emph{length edge} or \emph{width edge} of $T$ 
as in the following figure.
\begin{figure}[H]
\includegraphics[]{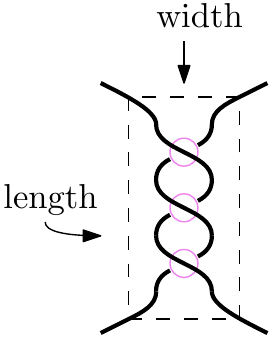}
\end{figure}
\end{definition}

\vskip10pt

\section{Surfaces in link complements.}
\vskip7pt
\subsection{Normal position.}\label{subsec: normal position} 
We are interested in studying compact surfaces $S$ properly embedded in 
$\bbS^3 \ssm \NN(L)$. If $\partial S \neq \emptyset$ we extend $S$ by 
shrinking the neighborhood $\NN(L)$ radially. This determines  a map 
$\iota:(S,\partial S) \to (\bbS^3,L)$,  whose image we denote by $S$ as well, 
which is an embedding on the interior of $S$.

\begin{lemma}\label{lem: normal form} 
Let $S\subset \bbS^3 \ssm \NN(L)$ be a proper surface with no meridional boundary 
components, and let $(\bfT,t)$ be a twist box. Then,  up to isotopy,
$S \cap \bfT$ is a disjoint  union of disks  $D \subset (\bfT,t)$ of one of the 
following three types:
\vskip7pt
\begin{enumerate}
    \item[ \underline{Type 0}:] $D$ separates the two strings of $t$.
    \vskip7pt
    \item[ \underline{Type 1}:] $\partial D$ decomposes as the union of two arcs 
    $\alpha \cup \beta$ 
    such that $\alpha\subset t$ and $\beta\subset \partial \bfT$.
    \vskip7pt
    \item[ \underline{Type 2}:] $\partial D$ decomposes as the union of four arcs 
    $\alpha_1 \cup \beta_1 \cup \alpha_2 \cup \beta_2$ where $\alpha_i \subset t_i$ and 
    $\beta_i \subset \partial \bfT$.
\end{enumerate}
Moreover, the isotopy decreases the number of bubbles that $S$ meets and   
we may further assume that $\iota|_{\partial L}:\partial S \to L$ is a covering map. 
\end{lemma}

\begin{proof}
If no component of $\partial S$ is a meridian, we may assume that up to isotopy 
$\iota|_{\partial L}:\partial S \to L$ is a covering map. 

The twist box $(\bfT,t)$ is a trivial 2-tangle. The complement $\bfT\ssm \NN(t)$ can be 
identified with $P\times [0,1]$ where $P$ is a twice holed disk.  Let $E$ be the disk
$\alpha\times [0,1]$ where $\alpha$ is the simple arc connecting the two holes of $P$.
Up to a small isotopy, we may assume that $S$ intersects $E$ transversely. 
Since the bubbles in $\bfT$ are in some neighborhood of $E$, we may assume that $S$ meets 
a bubble if it does so in $E$. The intersection $S\cap E$ comprises of simple closed 
curves and arcs. All curves and arcs except those connecting $\alpha\times \{0\}$ to 
$\alpha\times \{1\}$ can be eliminated by an isotopy pushing $S$ off $\bfT$. 
This isotopy decreases the number of bubbles $S$ meets. The number of bubbles the 
resulting surface meets equals the number of such arcs times the number of crossings 
in the corresponding twist region.

Up to isotopy, we may also assume that $S$ intersects $P\times \{ \tfrac{1}{2}\}$ 
transversely. Hence, $S\cap (P\times \{ \tfrac{1}{2}\})$ is a collection of simple 
closed curves and arcs. By pushing $S$ outwards towards the boundary of the disk $P$, 
one can assume that each component of $S\cap (P\times \{ \tfrac{1}{2}\})$ is of the 
following form:

\begin{enumerate} \setcounter{enumi}{-1}
    \item An arc connecting the boundary of the disk $P$ to itself separating the holes, 
    and intersecting $\alpha$ once.
    \vskip5pt
    \item An arc connecting a hole to the boundary of the disk and not intersecting $\alpha$.
    \vskip5pt
    \item An arc connecting the two holes and not intersecting $\alpha$.
\end{enumerate}
\vskip5pt
Thus, $S\cap (P\times [\tfrac{1}{2}-\varepsilon,\tfrac{1}{2}+\varepsilon])$ 
is a collection of disks of types (0),(1) or (2) as stated. By an ambient 
isotopy, we can stretch the slab
$P\times [\tfrac{1}{2}-\varepsilon,\tfrac{1}{2}+\varepsilon]$ to $P\times [0,1]=T$.
The number of bubbles the resulting surface meets equals the number of arcs of 
type (0) times the number of twist in the twist region. The arcs of type (0) 
are in one-to-one correspondence with the arcs of $S\cap E$.  Note that the fact 
that $\iota:\partial S \to L$ is a covering map was not  affected by the isotopies above.
\end{proof}

\begin{definition}\label{def: normal position}
A surface $S\subset \bbS^3 \ssm \NN(L)$ is \emph{in normal position} if its 
extension intersects the planes $P^\pm$ transversely and also intersects
each twist box  as specified in Lemma \ref{lem: normal form} and 
$\iota|_{\partial S}:\partial S \to L$ is a covering map. In particular, 
$S$ has no meridional boundary components.
\end{definition}

\begin{lemma}
Let $S \subset \bbS^3 \ssm \calN (L)$ be a surface in normal position, and let 
$(\bfT,t)$ be a twist box. Then, up to isotopy, each component of the intersection 
$S \cap \bfT \cap P^\pm$  looks as in Figure \ref{fig: three types of intersection}.
\qed
\end{lemma}

\begin{figure}
   \centering
    \begin{overpic}[height=3.5cm]{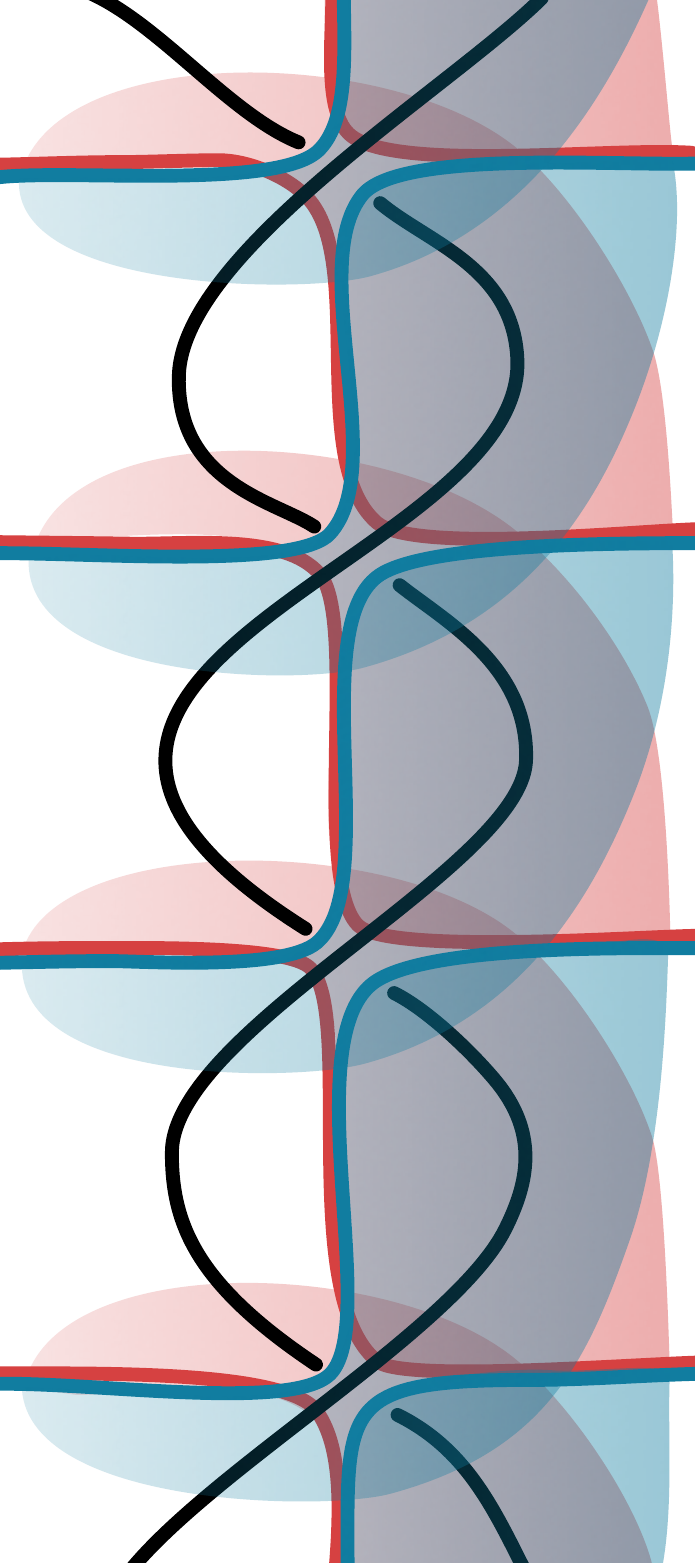}
    \put(5,-10){Type 0}
    \end{overpic}
        \hskip 1cm
    \begin{overpic}[height=3.5cm]{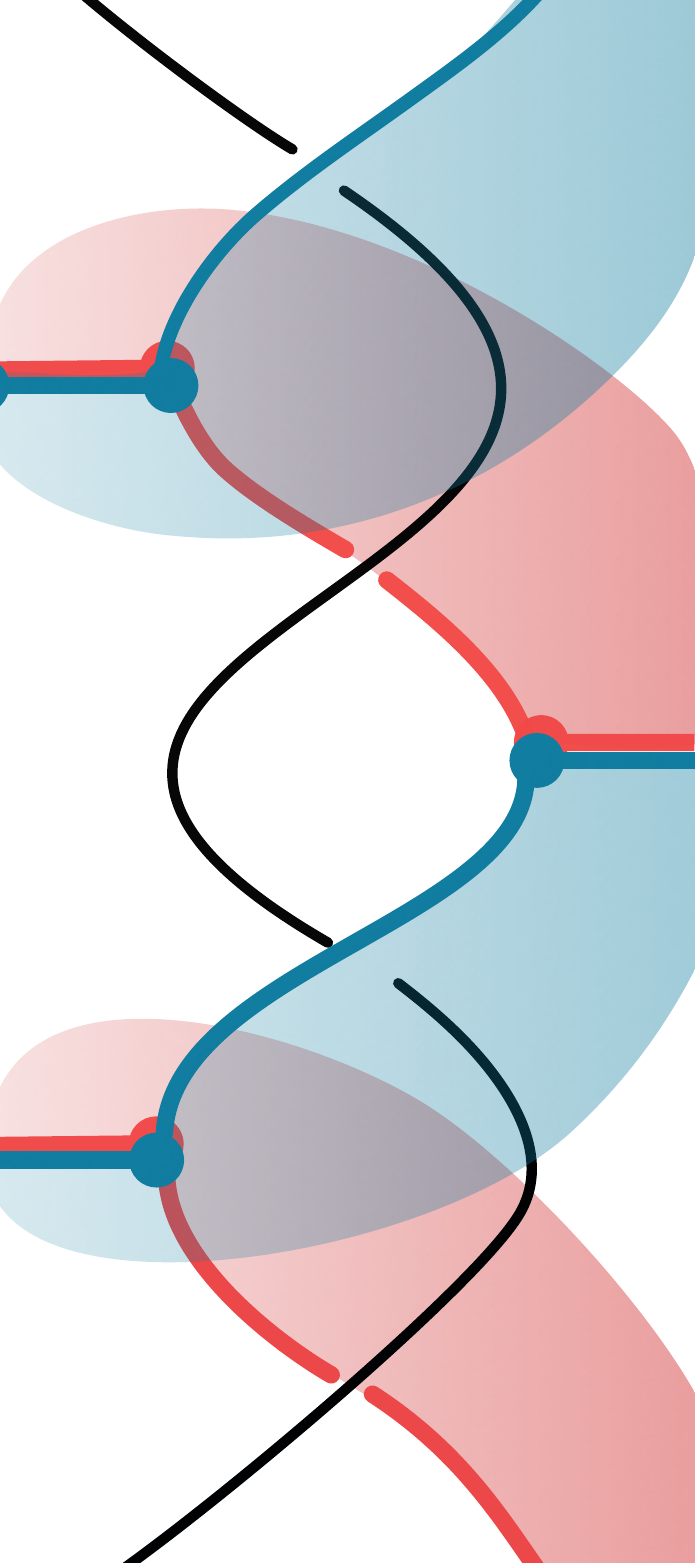}
    \put(5,-10){Type 1}
    \end{overpic}
        \hskip 1cm
    \begin{overpic}[height=3.5cm]{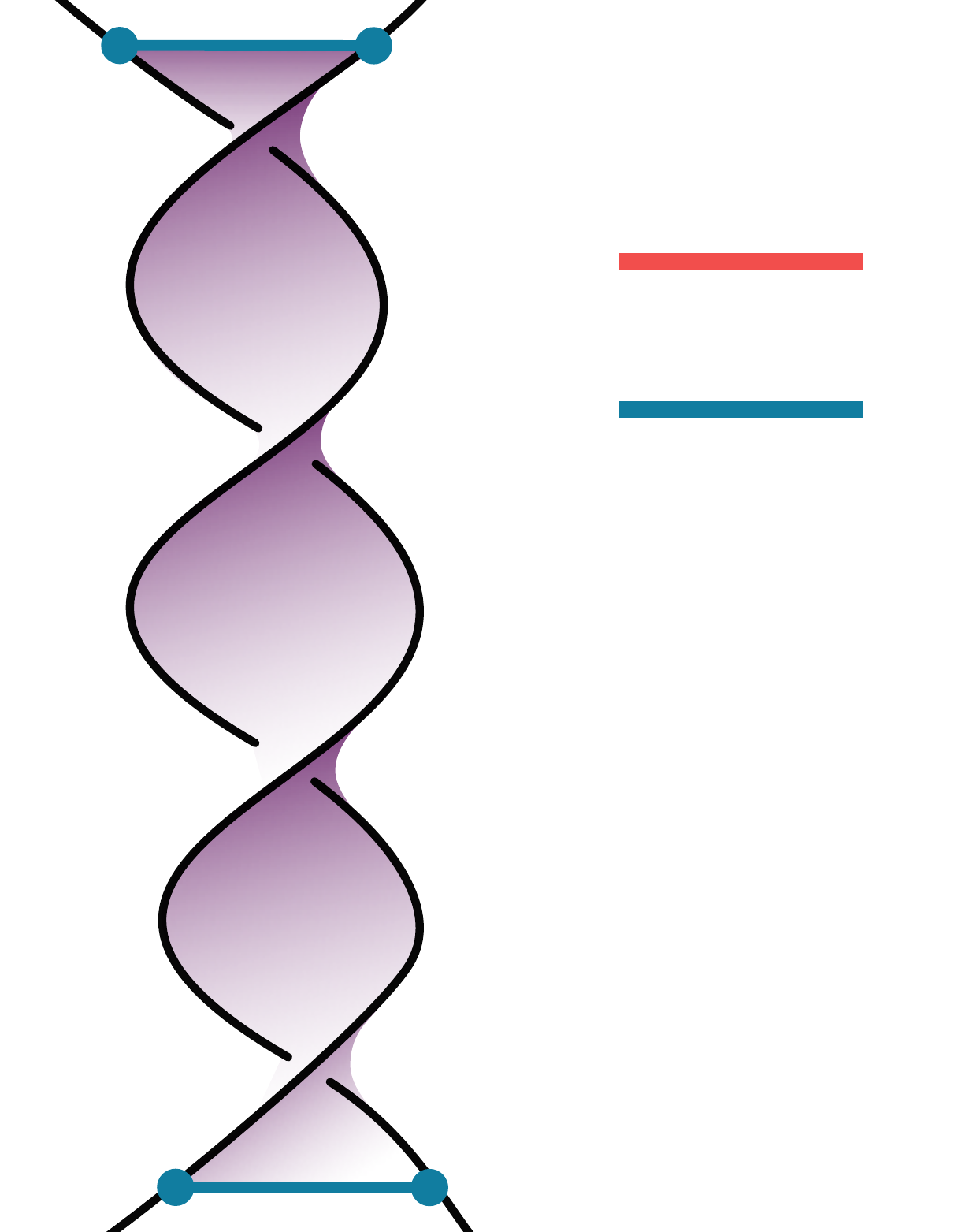}
    \put(5,-10){Type 2}
    \put(75,76){$S\cap P^-$}
    \put(75,63){$S\cap P^+$}
    \end{overpic}
    \medskip
    \caption{The possible three types of intersection of $S$ with a twist box.}
    \label{fig: three types of intersection}
\end{figure}

\subsection{Curves of intersection.}\label{subsec: Curves of intersection}
Let $S \subseteq \bbS^3 \ssm \NN(L)$ be a surface in normal position. 
We would like to study the surface $S$ through its curves of intersection 
with the planes $P^\pm$. 

Let $\calT$ be the union of all twist boxes of  $L$. Consider the collection 
of disks $\calD$ of Type (2) which occur as intersections $S\cap \calT$. 
We may assume that $\partial \calD \subset P \cup L$, and that the subsurface 
$\widehat{S}=S \ssm \calD$ is transversal to $P^\pm$.

Recall the map $\iota:(S,\partial S) \to (\bbS^3,L)$. 
Define $\calC^+ = \partial \iota\ii(\widehat{S} \cap H^+)$ and 
$\calC^- =  \partial  \iota\ii(\widehat{S} \cap H^-)$.
Now define $\calC =\calC^+ \cup \calC^-$.  As each of  $P^\pm$ is a 
$2$-sphere, $\widehat{S}\cap H^\pm$ is a 
collection of subsurfaces of $\widehat{S}$, the boundary of which are simple 
closed curves $c \subset S$. For $c\in \calC^+$, denote by $S_c$ the component 
of $\widehat{S} \cap H^+$ so that $c\subset \partial S_c$, and respectively for 
$c\in \calC^-$.

We think of curves in $\calC^\pm$ as curves on $P^\pm$, as they are disjoint outside $L$.
Here, and in most of the figures in the remainder of this paper, curves in $\calC^+$ 
are colored blue while curves in $\calC^-$ are colored orange.

Assume $S$ is in normal position, and let $\bfT$ be a twist box. The curves of 
intersection $\calC$ of $S$ which meet a connected component of $\iota(S) \cap \bfT$ 
must meet the corresponding twist region $T$ in one of the following three configurations:

    \vskip5pt

\begin{figure}[H] 
    \includegraphics[width = 0.5\textwidth]{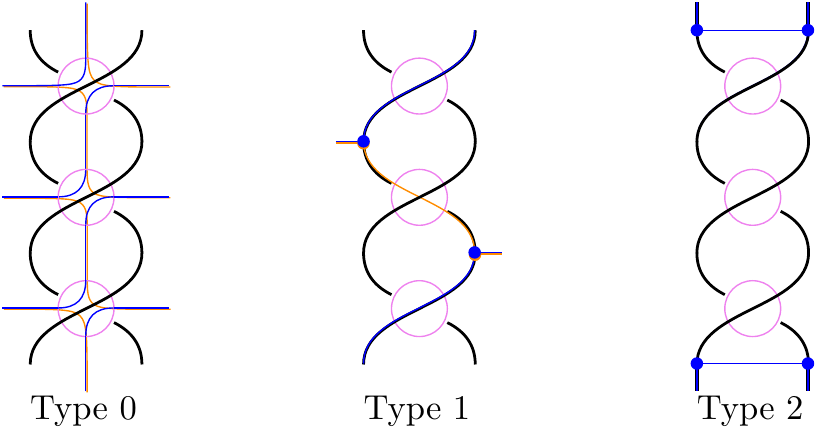}
\end{figure}

 In order to annalize  the curves  $c\in \calC$  we need to consider specific subarcs 
 and points of $c$ which we define next.

 \vskip5pt

\begin{definition}\label{def: c passes}\label{def: numbers}\label{def: joints}
For a curve $c\in \calC$, we define the following arcs and points (they are illustrated in 
the figures following the definition):
\begin{enumerate}
    \item A \emph{\I-joint} (``interior-joint'') of $c$ is a subarc of $c$ which is a 
    connected component of $c \cap \iota\ii(B)$ for some bubble $B$. The number of \I-joints 
    of $c$ is denoted $\bbl{c}$.
    \vskip5pt
    \item A \emph{\D-joint} (``boundary-joint'') of $c$ is an endpoint of a connected 
    component of $c\cap \partial S$. The number of \D-joints of $c$ is denoted $\bdr{c}$
    \vskip5pt
    \item A \emph{\joint} of $c$ is a \I-joint or a \D-joint of $c$. The number of joints 
    of $c$ is denoted by $\bbb{c}=\bbl{c}+\bdr{c}$.
    \vskip5pt   
    \item Define $\calC_{i,j} = \{ c\in\calC \mid \bbl{c}=i,\bdr{c}=j\}$.
    \vskip5pt
    \item A \emph{bone} of $c$ is a connected component of $c$ minus its joints. 
    Note that all bones are arcs which are mapped by $\iota$ to $P$.
    \vskip5pt
    \item A \emph{\D-bone} of $c$ is a bone which is contained in $\partial S$. 
    Note that the endpoints of \D-bones are \D-joints. All other bones are \emph{\I-bones}.
    \vskip5pt
    \item A \emph{limb}  of $c$ is a subarc $\alpha \subset c$ with endpoints in the 
    interiors of bones. Two limbs are equal if there is an isotopy of limbs (in $c$) 
    between them. In particular, their endpoints lie in the interior of the same 
    bones.  The quantities $\bbb{\alpha}$, $\bbl{\alpha}$ and $\bdr{\alpha}$ are 
    defined as for curves.
    \vskip5pt
    \item A \emph{turn} of $c$ is a limb of $c$ that contains exactly one joint, 
    this joint is an \I-joint and the endpoints of the limb are outside 
    twist regions.
    A curve \emph{turns} at a twist region if it contains a turn in that region. 
    \vskip5pt
    \item A \emph{wiggle} of $c$ is a limb of $c$ that contains exactly two joints, 
    these joints are \I-joints through consecutive bubbles of a twist box, and 
    the endpoints of the limb are outside the twist regions.
    A curve \emph{wiggles} through a twist region if it contains a wiggle in that region. 
    
    \begin{figure}[H]
    \includegraphics{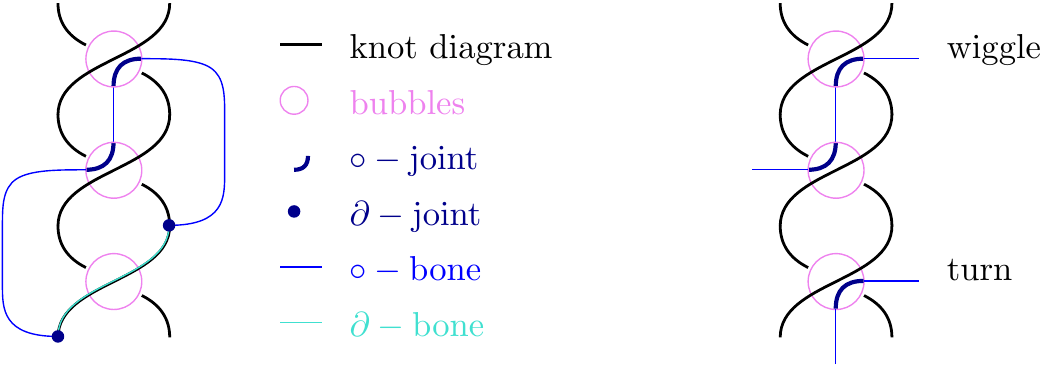}
    \end{figure} 
    
    \item Let $\calB$ be a 3-ball bounded by a bubble $B$, then the components of 
    $S\cap \calB$ are \emph{saddles}. Those are disks whose boundary is 
    $\iota(\alpha_1^+\cup \alpha_2^+ \cup \alpha_1^- \cup \alpha_2^-)$ where 
    $\alpha_i^\pm$ are \I-joints of curves in $\calC^\pm$ respectively. The 
    two \I-joints $\alpha_1^+,\alpha_2^+$ (and the two \I-joints $\alpha_1^-,\alpha_2^-$) 
    are said to be \emph{opposite}.
    
    \vskip7pt
    
    \begin{figure}[H]
    \includegraphics[width = 0.6 \textwidth]{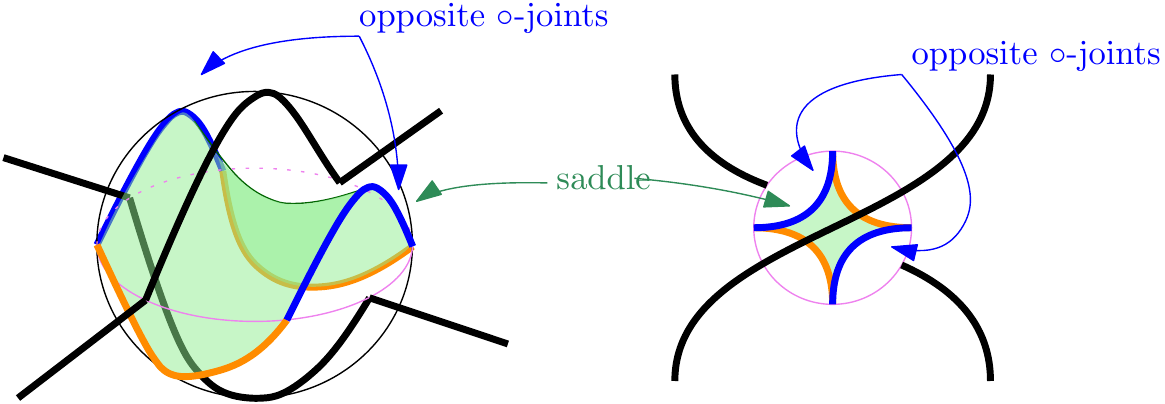}
    \caption{A saddle of $S$ in a bubble viewed from an angle and from the top.} 
\end{figure}

\end{enumerate}

\end{definition}

\begin{remark}\label{rem: joints and bones}
Note that a \D-joint connects a \D-bone and \I-bone, while a \I-joint connects 
two \I-bones. The \D-joints and \D-bones are contained in the boundary of $S$, 
while \I-joints and \I-bones are contained in the interior of $S$. 
\end{remark}

\begin{definition}
Two curves (or limbs of curves) $c,c'\in\calC$ are \emph{abutting} if they share 
an \I-bone and $c\ne c'$. Necessarily, if $c\in\calC^+$ then $c'\in\calC^-$ and vice versa. 
\end{definition}
The following figure shows an example of two abutting curves $c$ and $c'$.

\begin{figure}[H]
    \includegraphics[width = 0.32 \textwidth]{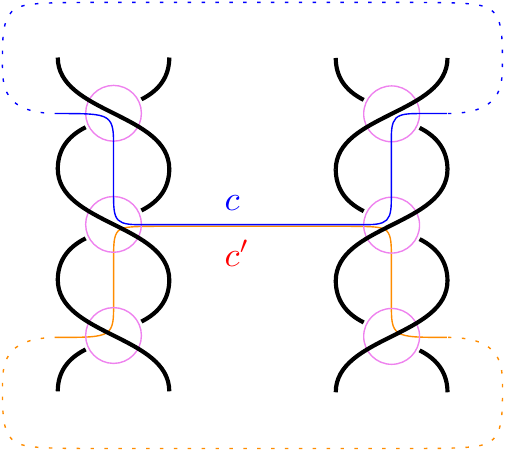}
\end{figure}

\subsection{Taut surfaces}\label{subsec: taut surfaces}
\begin{definition}\label{def: complexity} Given an incompressible surface 
$S \subset \bbS^3 \ssm \NN(L)$
we define a { \it lexicographic complexity} of $S$ as follows: 
\begin{equation}
\text{Com}(S) = (\sum_{c\in\calC} \bbl{c},\, 
\sum_{c\in\calC} \bdr{c},\, \abs{\mathcal{C}})
\end{equation}
\end{definition}

Recall that a  properly embedded surface $S$ in a $3$-manifold $M$ is called 
\emph{essential} if it is either  a 2-sphere which does not bound a 3-ball, 
or it is incompressible, boundary incompressible and not boundary parallel.

\begin{definition}
Let $S\subset \bbS^3 \ssm \NN(L)$ be an essential surface in normal position. 
The surface $S$ is \emph{taut} if
either
\begin{enumerate}[label=(\roman*)]
    \item $S$ is an essential 2-sphere, and $S$ minimizes complexity among all 
    essential 2-spheres, or
    \item $S$ is not a 2-sphere, the link $L$ is not split (i.e., $\bbS^3 \ssm L$ 
    is irreducible), and 
    $S$ minimizes complexity in its isotopy class.
\end{enumerate}
\end{definition}

The next lemma shows that the intersection curves of taut surfaces must have certain properties. 

\begin{lemma}\label{lem: properties of curves}
Assume that the diagram $D(L)$ is connected. 
Let $S\subset \bbS^3 \ssm \NN(L)$ be a taut surface. Then, for all $c \in \calC$ we have:
\vskip5pt
\begin{enumerate}
    \item $S_c$ is a disk.
    \vskip5pt
    \item $\bdr{c}$ is even. 
    \vskip5pt
    \item If $\bdr{c}\le 2$ then $\bbl{c} >0$.
    \vskip5pt
   \item  If $\bdr{c}=0$ then $\bbl{c}$ is even.
   \vskip5pt
   \item If a curve $c$ meets a bubble $B$ more than once, then 
   it does so in two opposite \I-joints.
   \vskip5pt
   \item If a curve $c$ has two \D-joints on a connected component of $P^+ \cap L$ 
   (or $P^- \cap L$), then they are the endpoints of a \D-bone. Moreover, the two 
   \I-bones incident to them are in different  regions of $P\ssm D(L)$.
   \vskip5pt
   \item The curve $c$ is not a curve in $\calC_{1,2}$ as depicted in the following figure:
   \begin{figure}[H]
       \includegraphics[width = 0.12 \textwidth]{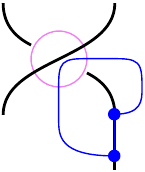}
   \end{figure}
  
\end{enumerate}
\end{lemma}

\begin{proof} Let $S\subset \bbS^3 \ssm \NN(L)$ be an essential surface 
satisfying  (i) or (ii). Note that in both cases, compressing along a disk 
$D\subset \bbS^3 \ssm \NN(L)$ with $D\cap S = \partial D$ results either 
in two essential spheres, or a surface in the same isotopy class of $S$. 
Thus, by the assumption on $S$, surfaces obtained by such a compression 
cannot have lower complexity.
\vskip5pt

\noindent (1) Since $S$ is essential, each subsurface $S_c$ must be planar, 
as otherwise it contains a non-trivial compression disk. If $S_c$ has more 
than one boundary component then compressing along a disk in $H^+$ or $H^-$ 
whose boundary separates boundary components of $S_c$ will result in a surface 
with fewer  intersections  with $P$ in contradiction to the choice of $S$.
\vskip5pt

\noindent (2) By definition, $\bdr{c}$ is the number of endpoints of arcs
in $c\ssm L$. Since each arc has two endpoints, $\bdr{c}$ is even.
\vskip5pt

\noindent (3) By (2) $\bdr{c}$ is either two or zero. If $\bdr{c} = 0$ and 
$\bbl{c} = 0$ then $c$ bounds a disk on $P \ssm L$. Compressing $S$ along 
this disk reduces the number of intersections with $P$. 

If $\bdr{c} = 2$ and $\bbl{c} = 0$ then $c$ bounds a disk $D$ in $P$ such that 
$\partial D = \alpha \cup \beta$,  where $\alpha$ is a \D-bone of $c$ and $\beta$ 
is an \I-joint of $c$. However, this is impossible by (2).

\vskip3pt

\noindent (4) The diagram $D(L)$ is a 4-regular graph, and thus it partitions $P$ 
into regions which can be given a checkerboard coloring, i.e., can be colored black 
and white so that two adjacent regions are colored in different colors. Consider 
the colors of complementary regions of $P \ssm \mathcal{T}$  which the curve $c$ 
intersects.  If $\bdr{c}=0$ every change of colors, of these regions along $c$, 
accounts for one bubble that $c$ meets. Since $c$ is a closed curve the total 
number of color changes is even, and correspondingly $\bbl{c}$  is even.

\vskip3pt

\noindent (5) Without loss of generality assume that $c\in\calC^+$. The curve $c$ can 
meet $B$ in two \D-bones, an \I-joint and a \D-bone or two \I-joints. If $c$ meets 
$B$ in two \D-bones, then the disk $S_c$ contains an arc connecting the two \D-bones. 
This arc together with an arc on $\partial \NN (L)$ bounds (by an innermost argument) 
a compression disk for $S$. Compressing along this disk reduces the complexity of $S$.

If $c$ meets $B$ in an \I-joint and a \D-bone, then the disk $S_c$ contains an arc 
connecting the \I-joint and the \D-bone. This arc, together with an arc on $B$ 
bounds a disk. Isotoping $S$ through this disk reduces the complexity of $S$.

Thus, $c$ meets $B$ in two \I-joints. By Lemma 1(ii) of \cite{Menasco}, $c$ does 
not have two \I-joints in $B$ on the same side of $L$, i.e., as in the following figure.

   \begin{figure}[H]
       \includegraphics[width = 0.2\textwidth]{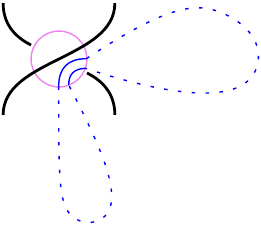}
   \end{figure}
    It follows that $c$ has at most two \I-joints in the same bubble, and 
    they are separated by $L$ in $B^+$. The number of components 
    of $S\cap B^+$ separating each of the  \I-joints of $c$ from $L$ 
    is the same: Each component of $S\cap B$ separating an \I-joint of $c$ and 
    $L$ belongs to a curve $c'$ in $\calC^+$. As curves in $\calC^+$ do not 
    intersect, in order to close up, $c'$ has to return to $B$ on the other 
    side of $L$ between $L$ and the other \I-joint of $c$ in $B$. This is 
    depicted in the following figure. 
    \begin{figure}[H]
       \includegraphics[width = 0.24\textwidth]{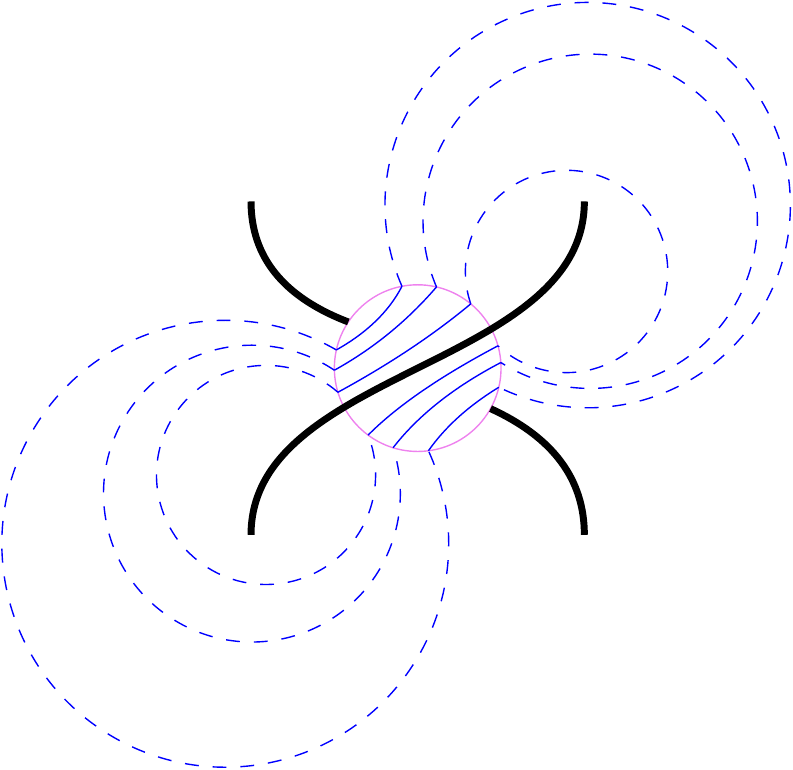}
   \end{figure}
   
   The intersection of $S$ with the ball bounded by $B$, is a finite collection 
   of stacked saddles, it follows that the \I-joints of $c$ belong to the same saddle.
\vskip5pt
    
\noindent (6) If $c$ has two \D-joints on the same component of $P^+ \cap L$. Let $\alpha$ 
in $P^+\cap L$ and $\beta$ in  $S_c$ be  arcs connecting the two \D-joints. By compressing 
along the disk bounded by $\alpha\cup \beta$ (using an innermost such disk) we reduce 
complexity unless the arc $\alpha$ is a \D-bone of $c$.

\vskip5pt

If $c$ contains a \D-bone $\alpha$ such that the two adjacent \I-joints are in the same 
region of $P\ssm D(L)$ then, by pushing $S$ through $P$ in a neighborhood of $\alpha$, we 
reduce the number of intersection points by 2, in contradiction to the minimal complexity 
of $S$.

\vskip5pt

\noindent (7) Assume in contradiction that $c$ is such a curve. By (6), any curve 
contained in $c$ has to be of a similar form. Assume $c$ is an innermost such curve. 
The curve $\bar c$ abutting $c$, as depicted in the following figure, meets the bubble 
twice, at a \D-bone and an \I-joint, in contradiction to (5).
\begin{figure}[H]
    \centering
    \includegraphics[width = 0.12 \textwidth]{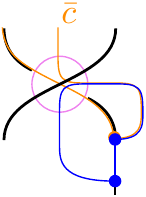}
\end{figure}

\end{proof}

\begin{remark}\label{rem: Not all C}
Note that if $S$ is taut then 
    $\calC_{0,0}=\calC_{0,2}=\calC_{i,2k+1}=\calC_{2k+1,0}=\emptyset$ for all 
    $i,k\in\bbN\cup \{0\}$.
\end{remark}

\medskip

\section{Euler Characteristic and curves of intersection} \label{sec: Curves and the 
Euler Characteristic}

\vskip10pt
From now on we assume that the surface $S$ is taut.

\subsection{Distributing Euler characteristic among curves}

For each curve $c\in \calC$ we will define the \emph{contribution} of $c$, and show 
that the Euler characteristic of $S$ can be computed by summing up the contributions 
of curves $c\in\calC$. 

\begin{definition}\label{def: contribution}
The \emph{contribution} $\chi_+(c)$  of a curve $c\in\calC$ is defined by
$$\chi_+(c)=1 - \frac{1}{4}\bbb{c}.$$
\end{definition}

\begin{lemma}\label{lem: Euler characteristic from Euler contributions}
If $S \subset \bbS^3 \ssm \NN(L)$ is taut then $\chi(S)=\sum_{c\in\calC} \chi_+(c)$.
\end{lemma}

\begin{proof}
The union of the collection of all the curves $c\in\calC$ on $S$ is an embedded 
graph $X$ of $S$. The vertices $X^0$ of the graph $X$ are the \D-joints and the 
endpoints of \I-joints. The edges $X^1$ of the graph $X$ are the bones and the 
\I-joints. The graph $X$ partitions $S$ into disk regions of three types:

\begin{enumerate}
    \item  subsurfaces $S_c \subseteq \widehat{S} \cap H^\pm$ for $c\in \calC^\pm$, 
    \smallskip
    \item  saddles $R \subseteq \widehat{S}\cap \calB$ where $\calB$ is a 3-ball 
    bounded by a bubble, or
    \smallskip
    \item regions $D \subset S$ corresponding to Type (2) disks.
\end{enumerate}

In case (3), the regions $D$ are disks whose boundary consists of two arcs on $L$ 
and two edges of $X$. By collapsing  each such disk $D$ to one of the edges in $X$ 
we get a homotopic surface. By abuse of notation, we call it $S$, and call the 
corresponding graph $X$. Note that in the new surface, $\partial S \subset X$ 
and consists of circles comprised of \D-bones and \D-joints. Moreover, along every 
\I-bone there are two abutting curves.
It follows that 
\begin{align}\label{eq: chi S using the graph}
\begin{split}
    \chi(S) &= \chi(X) + \sum_{S' \subseteq S \cap H^\pm} \chi(S') + 
    \sum_{R \subseteq S\cap B} 
    \chi(R)\\
    &= |X^0| - |X^1| + \sum_{S' \subseteq S \cap H^\pm} \chi(S') + 
    \sum_{R \subseteq S\cap B} \chi(R).
\end{split}
\end{align}
We compute how each $c\in\calC$ contributes to each of the summands in 
\eqref{eq: chi S using the graph}: 

\medskip

\underline{The vertices of $X$.} Every curve 
$c\in\calC$ passes through $2\bbl{c}$ vertices of $X^0$ in the interior of 
$S$ (those are the endpoints of \I-joints it passes). Furthermore, it goes through 
$\bdr{c}$ vertices of $X^0$ in $\partial S$. Each of these vertices 
belongs to two (abutting) curves $c\in \calC$. Hence,

\begin{equation}
    |X^0| = \sum_{c\in\calC} (\bbl{c} + \tfrac{1}{2} \bdr{c})
\end{equation} 

\underline{The edges of $X$.} Every curve $c\in \calC$ passes through 
$2\bbl{c} + \bdr{c}$ edges in $X^1$. Note that every \I-joint edge and every 
\D-bone edge belongs to exactly one curve in $\calC$, while each \I-bone edge 
in belongs to two curves in $\calC$. Each \I-joint edge appears in exactly one 
curve $c$ and is counted once in $\bbl{c}$. Hence, the number of \I-joint edges 
is $\sum_{c\in \calC} \bbl{c}.$  Similarly  each \D-bone edge appears in exactly 
one curve $c$ and is counted twice in $\bdr{c}$. Hence, the number of \D-bone 
edges is $\sum_{c\in \calC} \tfrac{1}{2} \bdr{c}.$
Finally, each \I-bone edge in $c$ accounts for two vertices in $X^0$. So the number 
of \I-bone edges is equal to  
 $$\tfrac{1}{2} |X^0| =  
 \tfrac{1}{2}(\sum_{c\in\calC} (\bbl{c} + \tfrac{1}{2} \bdr{c})).$$
Adding these contributions together gives
\begin{equation}
    |X^1| = \sum_{c\in\calC}( \tfrac{3}{2}\bbl{c} + \tfrac{3}{4} \bdr{c}).
\end{equation} 

\underline{Regions $S' \subset S\cap H^\pm$.} To every curve $c\in\calC$ there is a disk 
$S_c \subseteq S\cap H^\pm$. Thus,
\begin{equation} 
    \sum_{S' \subseteq S \cap H^\pm} \chi(S) = \sum_{c\in\calC} 
    1.
\end{equation}

\underline{Saddle regions $R\subset S\cap \calB$.} Each \I-joint of a curve $c\in \calC$ 
belongs to the boundary such a region. And so each curve passes through the 
boundary of $\bbl{c}$ such regions. As each saddle  region has four \I-joint 
in its boundary, we have
\begin{equation}
    \sum_{R \subseteq S\cap B} \chi(R) = \sum_{c\in C} \tfrac{1}{4}\bbl{c}.
\end{equation}

Summing over all of the above we get,
\begin{align*}
\chi(S) &= |X^0| - |X^1| + \sum_{S' \subseteq S \cap H^\pm} \chi(S') + 
\sum_{R \subseteq S\cap B} 
\chi(R)\\
&= \sum_{c\in \calC} (1 - 
\tfrac{1}{4}(\bbl{c} + \bdr{c}))\\
&= \sum_{c\in \calC} \chi_+(c).
\end{align*}
\end{proof}

\section{Redistribution of Euler characteristic}
\smallskip

In this section we redistribute the positive contribution of the Euler characteristic of 
curves, $\chi_+$, so that after the redistribution each curve's contribution is non-positive.

We first characterize the curves of intersection that have a 
positive $\chi_+$. The characterization is done in the following lemma:

\begin{lemma}\label{lem: possible positive curves}
 Let $c\in\calC$ so that $\chi_+(c)>0$ (i.e., $\bbb{c}<4$) then $c\in \calC_{2,0}$ or 
 $c\in\calC_{1,2}$ and it is one of the following six forms (up to 
 isotopy).
 
\smallskip

 \begin{figure}[H]
 \includegraphics[width =\textwidth]{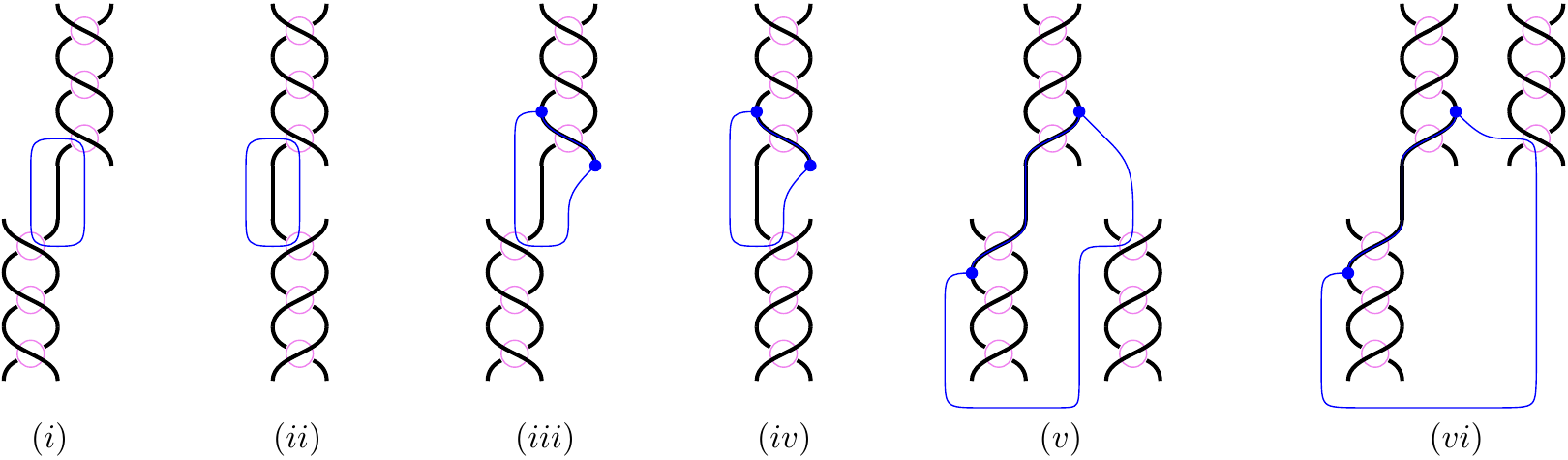}
 \caption{The six possibilities for a curve in 
 $c\in\calC_{2,0}\cup \calC_{1,2}$}\label{fig: positive curves}
 \end{figure}
\end{lemma}

\smallskip

Note that the curves $c$ in cases $(i)$ and  $(ii)$ are in $\calC_{2,0}$. 
The curves in cases $(iii), (iv), (v)$ and $(vi)$ belong to $\calC_{1,2}$

\begin{proof}
It follows from Definition \ref{def: contribution}, Lemma \ref{lem: properties of curves}
and Remark \ref{rem: Not all C} that if $\chi_+(c)>0$ then $c\in\calC_{2,0}$ 
or $c\in\calC_{1,2}$.  Since the diagram is prime, a curve $c\in\calC_{2,0}$  must 
contain two turns at different twist boxes. Hence $c$ is as depicted in figures (i) or (ii).

Let $c\in\calC_{1,2}$ and  let $\alpha$ denote a  limb which is a small extension 
of the unique \D-bone in $c$.  The endpoints of $\alpha$ must be in regions of 
$P\ssm D(L)$ of different color since the complementary limb of $\alpha$ in $c$ 
is a turn. Since $S$ is taut, Lemma \ref{lem: properties of curves}(7) implies 
that $\alpha$ must pass over at least one crossing of $D(L)$. Since the endpoints of 
$\alpha$ are in regions of different colors, $\alpha$ cannot connect the two regions 
adjacent to the two length edges of a twist region. Thus, $\alpha$ enters a twist 
region $T$ through its length edge, and exists on $L$. It can meet one or two twist 
regions. If $\alpha$ meets one twist region, then it must be as in figures (iii) or (iv). 
Otherwise, up to isotopy, it must be as in figures (v) or (vi).
\end{proof}

\begin{definition}
Denote by $\calC_{\pos}$ (resp. $\calC_{\le 0}$) the set of curves $c\in\calC$ 
such that $\chi_+(c)>0$ (resp. $\le 0 $). The lemma above shows that 
$\calC_{\pos} = \calC_{2,0} \cup \calC_{1,2}$. The \emph{type} of a curve in 
$\calC_{\pos}$ corresponds to the types of curves as depicted in 
Figure \ref{fig: positive curves}. For example, a curve of $\calC_{2,0}$ 
is of type (i) or (ii). 
\end{definition}

Next, we will describe a distinguished set, denoted by  $\calK$, of limbs 
of curves in $\calC$ to which we will ``reallocte'' some of the positive 
contribution of curves in $\calC_{\pos}$. We begin with a definition:

\begin{definition}
 An \emph{extremal} bubble is a first or last bubble of a twist region. A curve wiggles 
 \emph{extremally} if it wiggles through an extremal bubble. Assume that a curve or an 
 arc $\beta$ wiggles through a twist region extremally,  then the \I-bone $\alpha\subset \beta$ 
 which leaves the twist region from the \underline{extremal} bubble of the wiggle
 is called a \emph{core} of $\beta$. 
\end{definition}

\begin{definition}\label{def: curve types}
A \emph{vertebra} is an \I-bone $\tau$ in a curve $c$ connecting two turns of 
$c$ in two twist regions $T,T'$ so that $\tau$ meets the length edge of $T$ and 
the width edge of $T'$.

A \emph{rib} is a closed curve $c\in\calC_{4,0}$ which consists of exactly 
two turns and an extremal wiggle.
\end{definition}

\begin{remark}\label{rem: vertebrae of c20}
Note that the two \I-bones of a curve $c\in\calC_{2,0}$ of type (i) are vertebrae.
\end{remark}

\begin{lemma}\label{lem: onion}
Assume that $c_0$ is not a rib, and that $c_0$ contains a vertebra $\tau_0$.
Then, there exists a finite sequence of curves $c_0,c_1,\dots,c_n$, limbs
$\kappa_1,\dots,\kappa_n$ and bones $\tau_0,\dots,\tau_{n-1}$ such that:
\begin{enumerate}
    \item For $1\le i< n$, $c_i$ is a rib and $\tau_i$ is the \I-bone connecting its two turns.
    \vskip7pt
    \item For $1\le i\le n$, $\kappa_i$ is a limb of $c_i$ with a unique core $\tau_{i-1}$ 
        and $\bbb{\kappa_i}=3$. In particular, the curve $c_{i}$ abuts the curve $c_{i-1}$
        along the core $\tau_{i-1}$.
      \vskip7pt
    \item The curve $c_n$ has $\chi_+(c_n)<0$.
\end{enumerate}
    Moreover, given $\kappa_n$ one can uniquely determine the \I-bone $\tau_0$ and
    hence the curves $c_i$, arcs $\kappa_i$ and bones $\tau_i$  as above. 
\end{lemma}

\begin{definition}We will refer to the curves $c_i$ (resp. limbs $\kappa_i$) in the 
lemma as the \emph{layers curves} (resp. \emph{layer limbs}) of $\tau_0$ , and to 
the curve $c_n$ (resp. arc $\kappa_n$) as the \emph{terminal layer curve} 
(resp. \emph{terminal layer limb}) of $\tau_0$.
\end{definition}

\begin{proof}[Proof of Lemma \ref{lem: onion}]
We produce a sequence of curves, limbs and bones, satisfying the assumptions above, 
which terminates at the first curve $c_n$ such that $\chi_+(c_n)<0$.

Let $c_0$ and $\tau_0$ be as in the statement of the lemma. Let $c_1$ be the curve 
abutting $c_0$ along $\tau_0$. The bone $\tau_0$ connects an extremal wiggle and a 
turn of $c_1$, hence it is a core of $c_1$. Let $\kappa_1$ be the limb of $c_1$ 
containing $\tau_0$ and the adjacent wiggle and turn. If $\chi_+(c_1)<0$, then stop 
the process. Otherwise, $\bbb{c_1}=4$. Hence, the curve $c_1$ is a rib, i.e., it consists 
of a wiggle and two turns and has a unique core. Let $\alpha_1,\alpha_1'$ be the 
\I-joint of the two turns of $c_1$, and assume that $\alpha_1\subset \kappa_1$. 
Let $\tau_1$ be the \I-bone of $c_1$ connecting $\alpha_1,\alpha_1'$. The bone $\tau_1$ 
meets the length edge of the twist region containing $\alpha_1$.

Assume first that $\tau_1$ is not a vertebra of $c_1$, i.e., it meets the length 
edge of the twist region containing $\alpha_1'$. Then, the 
curve $c_2$ abutting $c_1$ along $\tau_1$ contains two wiggles in two different 
twist regions. It follows that $\chi_+(c_2)<0$ as otherwise $c_2$ bounds a twist 
reduction subdiagram. Let $\kappa_2$ be the limb in $c_2$, abutting $\tau_1$, 
consisting of a wiggle through the bubble of $\alpha_1$, and one more \I-joint 
at the bubble containing $\alpha_1'$. The bone $\tau_1$ is the unique core of 
the limb $\kappa_2$, and the process stops ($n=2$).

If $\tau_1$ is a vertebra, then we iterate the process. That is, we consider 
the curve $c_2$ abutting $c_1$ along $\tau_1$, and the limb $\kappa_2$ of $c_2$ 
containing $\tau_1$ and its adjacent wiggle and turn.

Since there are finitely many curves and limbs, the process either terminates 
or is periodic. It cannot be periodic because the initial curve $c_0$ is not a rib, 
but note that all the curves $c_i$ for $i<n$ are ribs.

Finally, given $\kappa_n$, the curve $c_{n-1}$ is the curve abutting 
$c_n$ along the unique core of $\kappa_n$. The curve $c_{n-1}$ has a unique 
core, which is also the core of a unique arc $\kappa_{n-1}$ (with 
$\bbb{\kappa_{n-1}}=3$). Repeating this process, we can retrace the 
sequence all the way to $\tau_0$.
\end{proof}

\begin{example}
In the figure below, we see two examples of outputs of the process 
in Lemma \ref{lem: onion}. Starting with the curve $c_0$ which is 
not a rib, and the vertebra $\tau_0$ of $c_0$, we get the curves
$c_0,c_1,c_2,c_3$, limbs $\kappa_1,\kappa_2,\kappa_3$ and \I-bones 
$\tau_0,\tau_1,\tau_2$. The limbs $\kappa_i$ are shown in bold the 
figure.  The \I-bones $\tau_0,\tau_1,\tau_2$ are the cores of 
$\kappa_1,\kappa_2,\kappa_3$ respectively.

\begin{figure}[H]
    \includegraphics{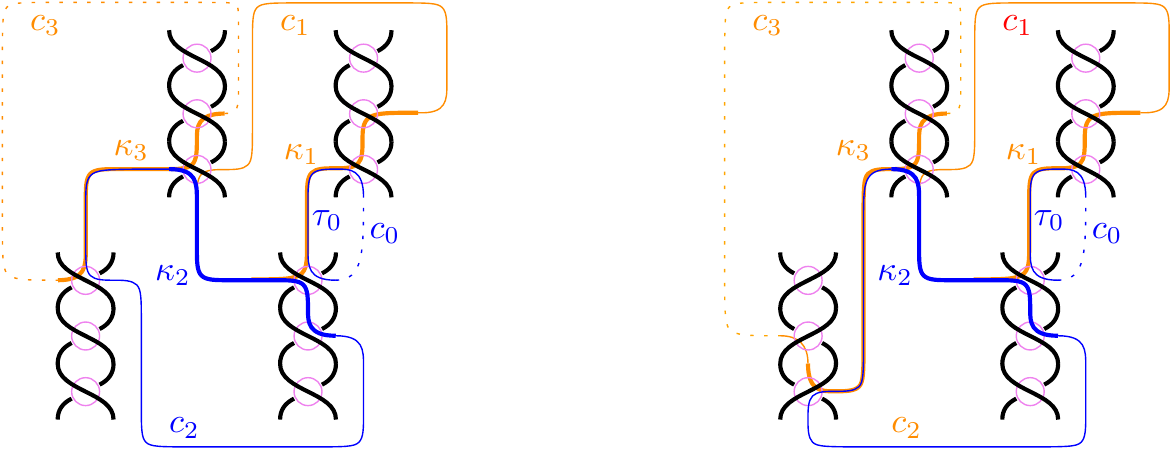}
\end{figure}

In the left figure, note that the \I-bone $\tau_2$ is a vertebra. If $\chi_+(c_3) < 0$ 
then the process stops at $c_3$ and $\kappa_3$ is its terminal limb. Otherwise, $c_3$ 
is again a rib, and the process continues.

In the right figure, the \I-bone connecting the two turns of $c_2$ is not a vertebra, 
as it meets the length edge of both twist regions. Therefore the limb $\kappa_3$, 
which is shown in bold in the figure, contains a wiggle (through the bubble in which 
both $c_1$ and $c_2$ turn) and ``half a wiggle'' (through the other bubble in which 
$c_2$ turns). As the proof shows, the curve $c_3$ necessarily satisfies $\chi_+(c_3)<0$, 
and thus  it is the terminal curve of the process.
\end{example}

\vskip3pt

\begin{definition}\label{def:Definition of K}  Definition of $\calK$. \hfill

\begin{enumerate}
    \item Each $c\in\calC_{2,0}$ of type (i) has two \I-bones. Each \I-bone $\tau$ 
    of $c$ is a vertebra and hence by Lemma \ref{lem: onion} determines a terminal 
    layer limb $\kappa_n$. Let $\calK_{3,0}$ be the set of all terminal layer limbs 
    associated with all \I-bones of curves in $c\in \calC_{2,0}$.

\vskip5pt

\item Each $c\in\calC_{2,0}$ of type (ii) determines an arc $\kappa$ which abuts $c$ 
and wiggles through the two twist regions. Let $\calK_{4,0}$ be the collection of all arcs 
$\kappa$ obtained in this way.

\vskip5pt

\item Each $c\in \calC_{2,1}$, determines an arc $\kappa$ which abuts $c$, wiggles through 
the twist region in which $c$ turns, and contains one of the \D-joints of $c$. Let $\calK_{2,1}$ 
be the collection of all arcs $\kappa$ obtained in this way.
\end{enumerate}

Finally, let $\calK$ be the set $ \calK_{3,0}\cup \calK_{4,0}\cup \calK_{2,1} $.
\end{definition}

\begin{remark}\label{rem: core retracing}
We note that given $\kappa \in \calK$ one can uniquely determine the curve
$c\in\calC_{\pos}$ that determines it. Conversely, to each curve 
in $\calC_{2,0}$ of type (i) there are two curves of $\calK_{3,0}$ corresponding 
to the  two choices of \I-bones of $c$. To each 
curve in $\calC_{1,2}$ and each curve in $\calC_{2,0}$ of type (ii) there is a 
unique arc in $\calK_{2,1}$ and $\calK_{4,0}$ respectively.  In particular, 
$\tfrac12 |\calK_{3,0}|  + |\calK_{4,0}| = |\calC_{2,0}|$ and $|\calK_{2,1}| =
|\calC_{1,2}|$.
\end{remark}

\begin{lemma}\label{lem: K are disjoint}
The arcs in $\calK$ which belong to curves in $\calC^+$ (resp. $\calC^-$) are 
pairwise disjoint. 
\end{lemma}

\begin{proof}
Assume that the two arcs $\kappa,\kappa' \in \calK$ meet. Since, by Remark 
\ref{rem: core retracing}, every arc in $\calK$ is determined by its core, 
it suffices to show that $\kappa$ and $\kappa'$ have the same core.

By assumption $\kappa,\kappa'$ share a joint, and hence are limbs of the same curve $c$. 
If this joint is a \D-joint then $\kappa,\kappa'\in\calK_{2,1}$ and their core is the 
unique \I-bone incident to their shared \D-joint. If the joint is an \I-joint that is 
part of a wiggle of $c$, then, since the diagram is 3-highly twisted, there is a 
unique core emanating from the extremal bubble of this wiggle, which is shared by 
both $\kappa$ and $\kappa'$. Finally, if the joint is an \I-joint that is part of a 
turn of $c$, then $\kappa,\kappa'\in\calK_{3,0}$. The cores of $\kappa,\kappa'$ must 
be the unique \I-bone which emanates from the width edge of the twist 
region in which the turn occurs because it is also the vertebra of the previous layer curve.
\end{proof}

\begin{lemma}\label{lem: gap of c containing kappa}
For all $\kappa\in \calK$ we have $\bbb{c} \ge \bbb{\kappa} + 2$ where $c$ is the 
unique curve in $\calC$ containing $\kappa$.
\end{lemma}

\begin{proof}
Let $\kappa\in \calK$ and let $c\in\calC$ be the curve containing $\kappa$. Assume 
for contradiction that $\bbb{c}< \bbb{\kappa} + 2$.

If $\kappa \in \calK_{4,0}$ then $\bbb{c} < \bbb{\kappa}+2$ implies that 
$c\in \calC_{4,0}$. If this is the case, then since $c$ wiggles through two twist 
regions, the projection of $c$ to $P$  gives a twist-reduction subdiagram in 
contradiction to the assumption on the diagram.

If $\kappa \in \calK_{2,1}$, then $\bbb{c}<\bbb{\kappa}+2$ implies that 
$c\in \calC_{2,2}$. 
If this is the case, then $\kappa$ abuts some curve $c_0\in\calC_{1,2}$ which could 
be one of the figures (iii)-(vi) in Lemma \ref{lem: possible positive curves}.
For case (iii), there are two possible configurations for the closed curve 
$c\in\calC_{2,2}$ containing $\kappa$, while for each of the cases (iv)-(vi) 
there is only one possible such curve. Thus, all possible cases for the curve 
$c$ are shown in the figure below.

\begin{figure}[H]
    \includegraphics[width = 0.9\textwidth]{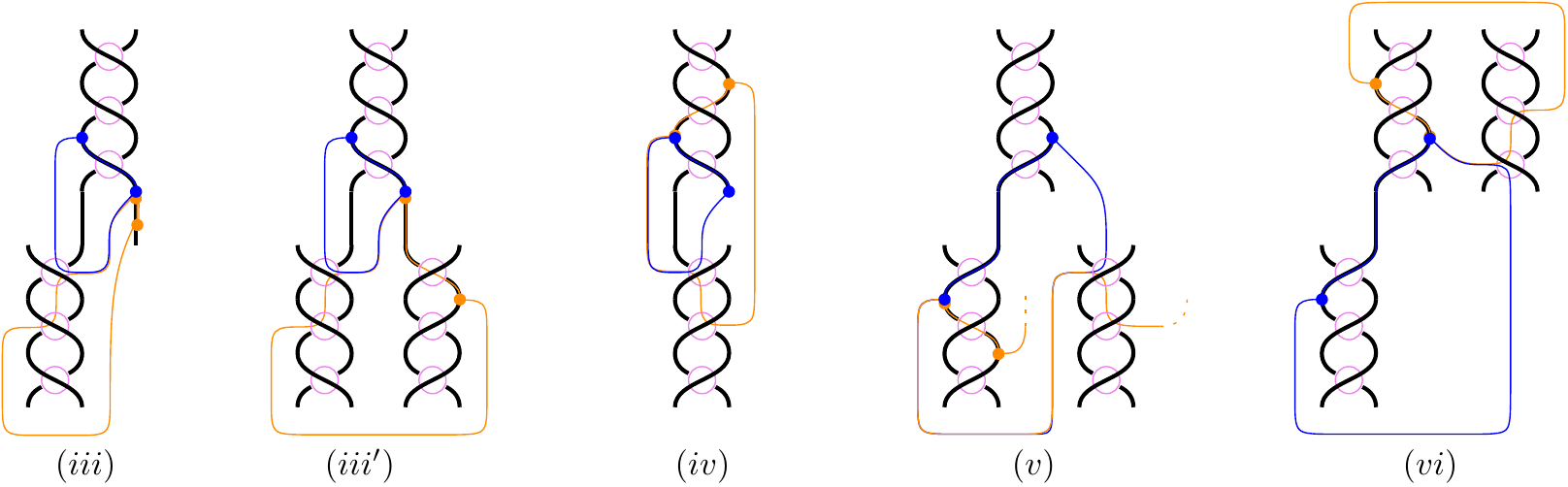}
\end{figure}

In case (iii) the surface is not taut and in all other cases, the curve $c$ bounds a 
twist-reduction subdiagram. 

If $\kappa \in \calK_{3,0}$, then $\bbb{c}<\bbb{\kappa}+2$ implies that $c\in\calC_{4,0}$, 
however this is in contradiction to the definition of $\calK_{3,0}$ given by 
Lemma~\ref{lem: onion}.
\end{proof}

\begin{remark}\label{rem: c and d are negative}
It follows from the proof of Lemma \ref{lem: onion} that if $c$ 
contains a limb $\kappa\in\calK_{3,0}$ such that the core of $\tau$ is not 
a vertebra of its abutting curve then $\bbb{c} \ge \bbb{k}+3$: Indeed, if $c=c_n$, 
$\kappa = \kappa_n$ and $\kappa$ abuts $c_{n-1}$ along a bone which is not a 
vertebra of $c_{n-1}$ then $c$ must have an additional \I-joint which is not 
contained in $\kappa$.
\end{remark}

\begin{definition}\label{def: distributing}
Let $c\in \calC$. If $c\in\calC_{\pos}$ define $\chi'(c)=0$. 
Otherwise, let $n_{3,0}$ (resp. $n_{4,0}$; $n_{2,1}$) be the number 
of limbs $\kappa \in \calK_{3,0}$ (resp. $\calK_{4,0}$; $ \calK_{2,1}$) in  $c$. 
We associate to $c$ the following quantity
\[
    \chi'(c) = \chi_+(c)+ \tfrac{1}{4}n_{3,0} + \tfrac{1}{2}n_{4,0} + \tfrac{1}{4} n_{2,1}.
 \]
\end{definition}

The next lemma shows that $\chi'$ is a redistribution of the Euler characteristic of $S$ 
among curves in $\calC_{\le 0}$. 
    
\begin{lemma}\label{lem: summing chi' gives Euler char}
$\chi(S) = \sum_{c\in \calC} \chi'(c).$
\end{lemma}

\begin{proof}
By Lemma \ref{lem: Euler characteristic from Euler contributions}, 
$\chi(S) = \sum_{c\in \calC} \chi_+ (c)$. Since $\chi'(c)=0$ for $c\in \calC_{\pos}$, 
$$\sum_{c\in\calC} \chi'(c) = \sum_{c\in \calC_{\np}} \chi'(c).$$
It thus remains to prove that 
$$\sum_{c\in \calC} \chi_+ (c) = \sum_{c\in \calC_{\np}} \chi'(c).$$
Subtracting $\sum_{c\in\calC_{\np}}\chi_+(c)$ from both sides and recalling that 
$\calC_{\np} = \calC \ssm \calC_{\pos}$, we have to show that
$$\sum_{c\in\calC_{\pos}} \chi_+(c) = \sum_{c'\in \calC_{\np}} (\chi'(c')-\chi_+(c')).$$
The left hand side is simply $\tfrac{1}{2}|\calC_{2,0}| + \tfrac{1}{4}|\calC_{1,2}|$ since 
$\calC_{\pos} = \calC_{2,0} \cup \calC_{1,2}$ and $$\chi_+(c)=\begin{cases} \tfrac{1}{2} & 
\mbox{if }c\in\calC_{2,0} \\ \tfrac{1}{4} & \mbox{if }c\in\calC_{1,2} \end{cases}.$$

By the definition of $\chi'$, the right hand side gives $\tfrac{1}{4}|\calK_{3,0}| +
\tfrac{1}{2}|\calK_{4,0}| +  \tfrac{1}{4}|\calK_{2,1}|$.
The proof is now complete by Remark \ref{rem: core retracing}.
\end{proof}

The next lemma shows that indeed $\chi'$ is non-positive.

\begin{lemma}\label{lem: non-positive contribution of chi'}
 $\chi'(c)\le 0$ for all $c\in\calC$.
\end{lemma}

\begin{proof}
If $c\in\calC_{\pos}$ then $\chi'(c)=0$. Let $c\in\calC_{\le 0}$ and let $n_{3,0},n_{4,0},n_{2,1}$ 
be as in Definition ~\ref{def: distributing} of $\chi'(c)$. 
By Lemma \ref{lem: K are disjoint}, the limbs of $\calK$ in $c$ are disjoint and therefore
\begin{equation}\label{eq: inequality from disjointness}
    \bbb{c} \ge \sum_{c\supset \kappa\in \calK} \bbb{\kappa}.
\end{equation}

Hence,
\begin{align}\label{eq: chi' inequality}
\begin{split}
\chi'(c)&=\chi_+(c)+\tfrac{1}{4}n_{3,0} + \tfrac{1}{2}n_{4,0} +\tfrac{1}{4}n_{2,1} 
 = (1-\tfrac{1}{4}\bbb{c}) +\tfrac{1}{4}n_{3,0} + \tfrac{1}{2}n_{4,0} +\tfrac{1}{4}n_{2,1} \\
& \le 1-\sum_{c\supset \kappa\in \calK} \tfrac{1}{4}\bbb{\kappa} + 
\tfrac{1}{4}n_{3,0} + \tfrac{1}{2}n_{4,0} +\tfrac{1}{4}n_{2,1} \\
& = 1 
+ \sum_{c\supset\kappa\in\calK_{3,0}} (\tfrac{1}{4} - \tfrac{1}{4}\bbb{\kappa})
+\sum_{c\supset\kappa\in\calK_{4,0}} (\tfrac{1}{2} - \tfrac{1}{4}\bbb{\kappa})
+ \sum_{c\supset\kappa\in\calK_{2,1}} (\tfrac{1}{4} - \tfrac{1}{4}\bbb{\kappa})\\
&=1 
+ \sum_{c\supset\kappa\in\calK_{3,0}} (\tfrac{1}{4} - \tfrac{1}{4}
\cdot 3)
+\sum_{c\supset\kappa\in\calK_{4,0}} (\tfrac{1}{2} - \tfrac{1}{4}\cdot 4)
+ \sum_{c\supset\kappa\in\calK_{2,1}} (\tfrac{1}{4} - \tfrac{1}{4}\cdot 3)\\
&= 1-\tfrac{1}{2}(n_{3,0}+n_{4,0}+n_{2,1}).
\end{split}
\end{align}

Now the argument is divided into cases depending on the sum $n=n_{3,0}+n_{4,0}+n_{2,1}$.

\vskip7pt

\textbf{Case 0. $n=0$.} We have $\chi'(c)=\chi_+(c)$. But since 
$c\in \calC_{\np}$ we have $\chi_+(c)\le 0$ and we are done.

\vskip7pt

\textbf{Case 1. $n=1$.} That is, $c$ contains a single subarc 
$\kappa\in \calK_{3,0} \cup \calK_{4,0} \cup \calK_{2,1}$. 
By Lemma \ref{lem: gap of c containing kappa}, $\bbb{c} \ge \bbb{\kappa}+2$. 
If $\kappa\in \calK_{4,0}$ we get $\bbb{c} \ge 6$ and thus
$\chi'(c) = 1-\tfrac14 \bbb{c} + \tfrac12 \le 0$.
Similarly, if $\kappa\in\calK_{3,0}\cup\calK_{2,1}$ we get 
$\bbb{c} \ge 5$ and thus $\chi'(c) = 1-\tfrac14 \bbb{c} + \tfrac14 \le 0$.

\vskip7pt

\textbf{Case 2. $n\ge 2$.} In this case we are done by 
inequality \eqref{eq: chi' inequality}.
\end{proof}

\begin{corollary}\label{cor: L is nonsplit} The link $\mathcal{L}$ is non-split nor the unknot.

\begin{proof} Assume, in contradiction,  that $\bbS^3\ssm \NN(\calL)$ has an 
essential sphere or a disk bounding a component of $L$. By Lemma \ref{lem: properties of curves} 
we may assume that $S$ is taut. Let $\calC$ be its curves of intersection With $P^{\pm}$. By 
Lemma \ref{lem: summing chi' gives Euler char}, 
$$0<\chi(S) = \sum_{c\in \calC} \chi'(c).$$ 
However, this contradicts Lemma \ref{lem: non-positive contribution of chi'} 
which states that $\chi'(c)\le 0$ for all $c\in\calC$.
\end{proof}
\end{corollary}

\begin{lemma}\label{lem: classification of chi'=0}
Let $S$ be a taut surface with $\chi(S)=0$, then any curve $c\in\calC$ is one of the following:
\begin{enumerate}
    \item $c\in\calC_{2,0} \cup \calC_{1,2}$;
    \vskip7pt
    \item $\bbb{c}=4$, i.e. $c\in\calC_{4,0} \cup\calC_{2,2} \cup \calC_{0,4}$;
    \vskip7pt
    \item $c$ contains a limb $\kappa\in\calK$ and has $\bbb{c}=\bbb{\kappa} +2$,
    \vskip7pt
    \item $c$ is the union of two limbs $\kappa_1,\kappa_2\in\calK$.
\end{enumerate}
Moreover, the cores of arcs in $\calK_{3,0}$ are vertebrae of their abutting curves.
\end{lemma}

\begin{proof}
By Lemma \ref{lem: summing chi' gives Euler char} and \ref{lem: non-positive contribution 
of chi'} each curve $c\in\calC$ must have $\chi'(c)=0$. It follow from the definition of 
$\chi'$ that the above are the only cases in which $\chi'(c)=0$. 

By Remark \ref{rem: c and d are negative} and Case 1 of the proof of 
Lemma \ref{lem: non-positive contribution of chi'}, if $c$ contains an arc 
$\kappa\in\calK_{3,0}$ whose core is not a vertebra of the abutting curve $c'$ 
then $\chi'(c)<0$.
\end{proof}

In the following table, we summarise the possible sets of curves with $\chi'=0$ 
and assign them names:

\begin{table}[H]\label{tab: table}
    \centering
    \begin{tabular}{c|c | c| c | c}
     Name of set & $\bbb\cdot$ & $\bbl\cdot$ & $\bdr\cdot$ & The set's composition/classification \\[0.85ex]
     \hline
     $\calC_{2,0}$ & 2 & 2 & 0 & type (i) or (ii)\\[0.85ex]
     $\calC_{1,2}$ & 3 & 1 & 2 & type (iii), (iv) or (v)\\[0.85ex]
     $\calC_{4,0}$ & 4 & 4 & 0 & 4 turns or 1 wiggle and 2 turns or 2 wiggles \\[0.85ex]
     $\calC_{2,2}$ & 4 & 2 & 2 & 2 \I-joints and 1 \D-bone\\[0.85ex]
     $\calC_{0,4}$ & 4 & 0 & 4 & 2 \D-bones\\[0.85ex]
     $\calK_{3,0} + 0,2$ & 5 & 3 & 2 &  1 limb of $\calK_{0,3}$ + \D-bone\\[0.85ex]
     $\calK_{2,1} + 1,1$ & 5 & 3 & 2 &  1 limb of $\calK_{2,1}$ + 1 \I-joint + 1 \D-joint\\[0.85ex]
     $\calK_{4,0} + 2,0$ & 6 & 6 & 0 &  1 limb of $\calK_{0,4}$ + 2 \I-joints\\[0.85ex]
     $\calK_{4,0} + 0,2$ & 6 & 4 & 2 &  1 limb of $\calK_{0,4}$ + \D-bone\\[0.85ex]
     $\calK_{3,0} + \calK_{3,0}$ & 6 & 6 & 0 &  2 limbs of $\calK_{3,0}$ = 2 wiggles + 2 turns.\\[0.85ex]
     $\calK_{2,1} + \calK_{2,1}$ & 6 & 4 & 2 &  2 limbs of $\calK_{2,1}$ =  2 wiggles + 
     \D-bone.\\[0.85ex]     
     $\calK_{4,0} + \calK_{4,0}$ & 8 & 8 & 0 &  2 limbs of $\calK_{4,0}$ =  4 wiggles.\\[0.75ex]
\end{tabular}
    \caption{Classification of all curves with $\chi'(c)=0$.}
    \label{tab: classification table}
\end{table}

\vskip8pt

\section{Atoroidal and unannular}\label{sec: Atoroidal and unannular}

In this section we prove that if the link $L$ has a diagram which satisfies the 
conditions of Theorem \ref{thm: highly twisted implies hyperbolic}, then its complement 
does not contain essential annuli or tori (in particular  $\chi(S)=0$). Let $S$ denote 
such a 2-torus or annulus. After an isotopy, if need be, we may assume that $S$ is taut.
We first need the following technical lemmas.
 
\vskip7pt

\begin{claim}\label{claim: special triple cases}
The following three configurations for curves $c \in \calC$ are impossible: 
\begin{figure}[H]
    \includegraphics[]{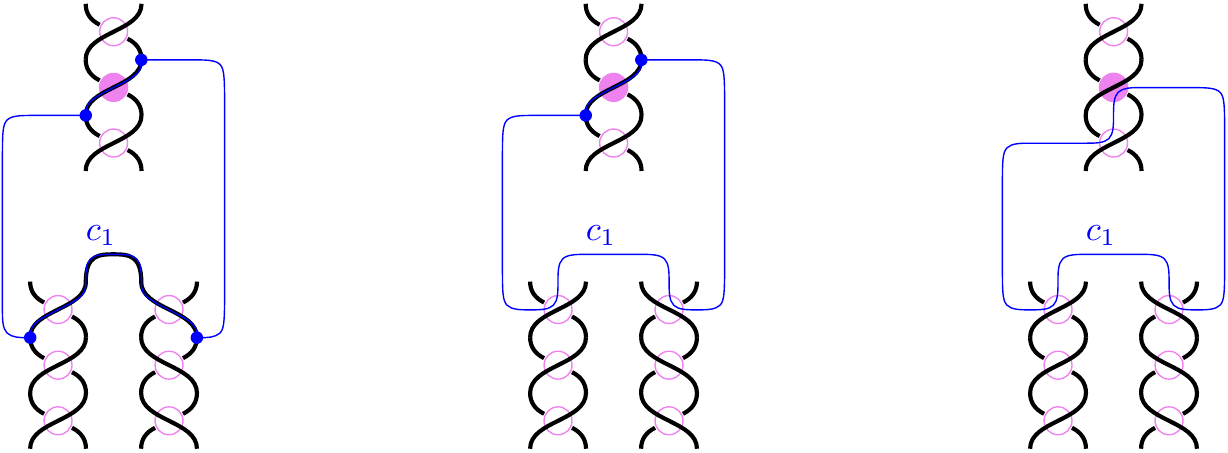}
\end{figure}
where the bubble  marked in pink is non-extremal.
\end{claim}

\begin{proof}
We argue simultaneously that the three configurations are impossible.
In each of these cases, let $c_2$ (marked in orange) be the depicted 
curve abutting $c_1$. 

\begin{figure}[H]
    \includegraphics[]{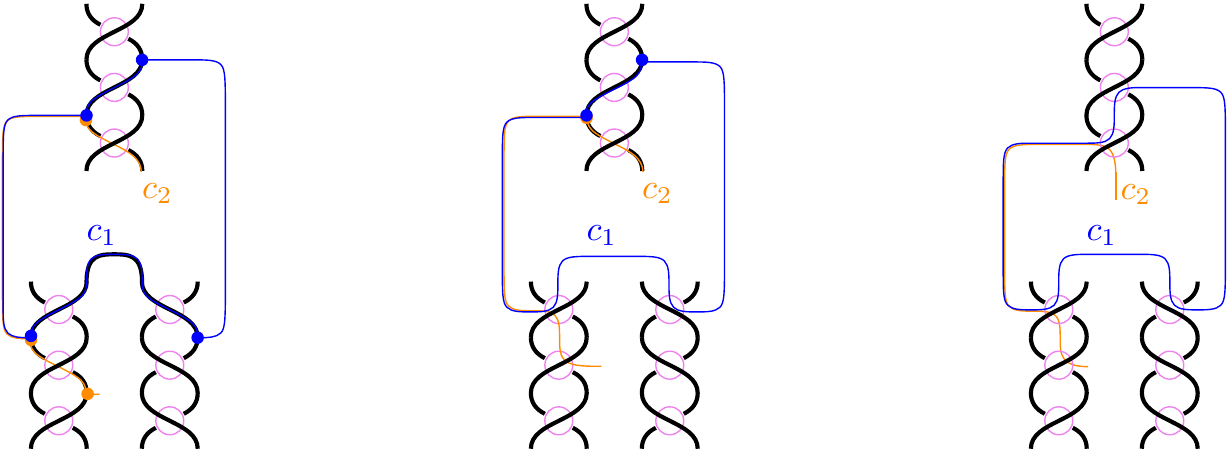}
\end{figure}
The limb of $c_2$ that is shown in the figure above has three joints. None of these 
joints belongs to a limb in $\calK$: in all cases, the curve $c_1$ is not in 
$\calC_{2,0}$ or $\calC_{1,2}$  nor a rib with a vertebra. By 
Lemma \ref{lem: classification of chi'=0}, it follows that $\bbb{c_2}=4$.  Thus the 
only way that $c_2$ can close up is depicted in in the following figures.

\begin{figure}[H]
    \includegraphics[]{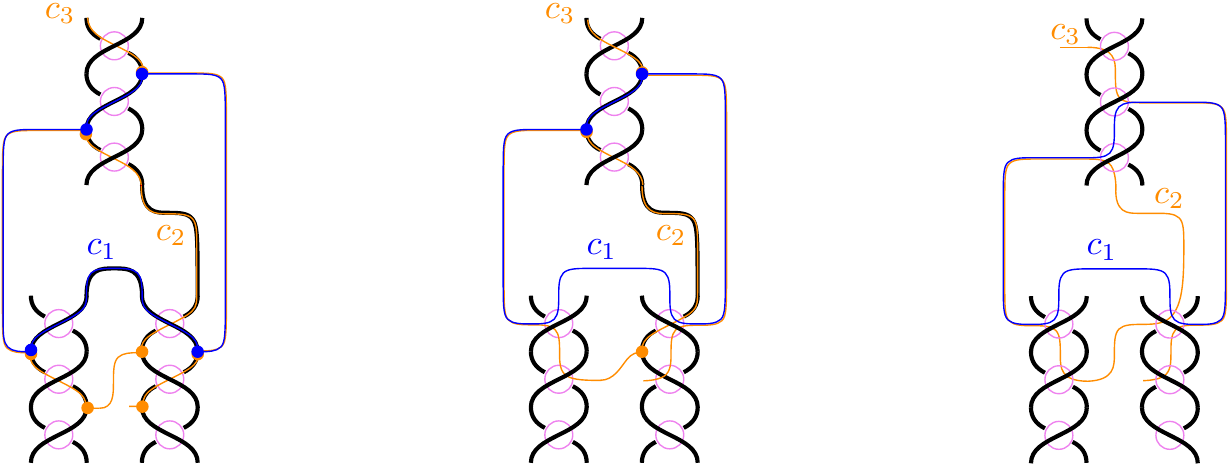}
\end{figure}

Let $c_3$ be the depicted curve abutting $c_1$. 
In the two left figures, one sees, as before, that $\bbb{c_3}=4$. In the figure on the 
right, $c_3$ is  in $\calK_{4,0}+{2,0}$ in the notation of Table \ref{tab: classification table}: 
Indeed, $c_3$ cannot have $\bbb{c_3}=4$, as otherwise it bounds a twist reduction subdiagram. 
Thus, $c_3$ must contain an arc of $\calK$. By elimination of the possibilities in 
Table \ref{tab: classification table}, $c_3$ must be in
$\calK_{4,0}+2,0$. Thus, in all cases, $c_3$ can close 
only after passing through an additional twist region as follows:

\begin{figure}[H]
    \includegraphics[]{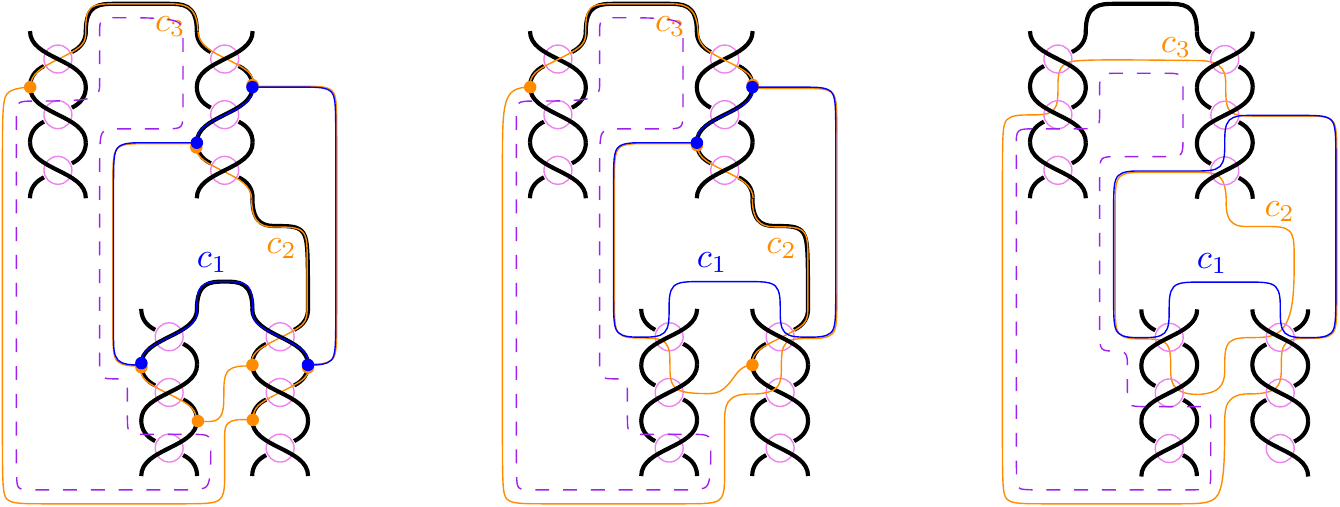}
\end{figure}

Thus, taking into account all possible configurations of the curves $c_2$ and $c_3$
determined by the stated configurations of the $c_1$ curves, the diagrams are seen to 
contained a closed curve depicted by the dashed curves in the figures. Each of 
the dashed curves bounds a twist-reduction subdiagram. This contradicts  the 
assumption that the diagrams are twist reduced  which finishes the proof of the claim.
\end{proof}

\begin{lemma}\label{lem: two turns}
If a curve $c\in\calC$ contains a bone $\tau$ connecting two turns of $c$ then one of 
the following holds:
\begin{enumerate}
    \vskip5pt
    \item the curve $c\in \calC_{2,0}$,
    \vskip5pt
    \item the curve $c$ is a rib, or
    \vskip5pt
    \item the \I-bone $\tau$ meets the width edge of both twist regions.
\end{enumerate}
In particular, $c$ cannot have three consecutive turns.
\end{lemma}

\begin{proof}
Let $\tau$ be a bone connecting two turns of $c$ it is therefore an \I-bone. Assume in  
contradiction that $\tau$ and $c$ do not satisfy neither of (1) - (3) of the lemma. 
That is, $c$ is not a curve in $\calC_{2,0}$ nor a rib, and $\tau$ does not meet 
the width edge of both twist regions. There are two cases to consider depending 
on whether $\tau$ meets a width edge or not.

If $\tau$ meets a width edge, then $\tau$ is a vertebra, i.e., it meets the length edge of one twist 
box and the width of the other. If this occurs set $c_0=c,\tau_0=\tau$. Since $c$ is assumed not 
to be a rib, it follows from Lemma \ref{lem: onion} that there exist curves $c_1,\dots,c_n$, limbs
$\kappa_1,\dots,\kappa_n$, and vertebrae $\tau_1,\dots,\tau_{n-1}$ so that the terminal layer 
$c_n$ has $\chi_+(c_n)<0$.  Moreover, note that the limb $\kappa_n$ is not in $\calK_{3,0}$ 
as otherwise by the uniqueness property, assured  in Lemma \ref{lem: onion}, the ``initial'' 
layer curve $c_0$ must be a curve in $\calC_{2,0}$ of type (i). This implies that $\kappa_n$ does 
not meet any limb of $\calK$.  As otherwise,  as in the  proof of Lemma \ref{lem: K are disjoint}, 
one can prove that $\kappa_n$ and the limb it meets must be equal. However, by Lemma 
\ref{lem: classification of chi'=0}, there is no curve, with $\chi'=0$, which has three \I-joints 
that do not belong to a limb of $\calK$.

If $\tau$ does not meet a width edge, then $\tau$ meets the length edge of both twist regions. 
It follows that the curve $c'$ abutting $c$ along $\tau$ has two wiggles which are connected 
by $\tau$. The curve $c'$ cannot be in $\calC_{4,0}$ as otherwise it bounds a twist-reduction 
subdiagram. By Lemma \ref{lem: classification of chi'=0}, one of the wiggles must meet a 
limb $\kappa\in \calK$. Since $\tau$ is a core of $c'$, it must be the core $\kappa$. It follows 
that $\kappa\in \calK_{4,0}$ and that $c\in\calC_{2,0}$ is of type (ii), in contradiction to our 
assumption.

In both cases, whether $\tau$ meets a width edge or not, we arrived to a contradiction. Hence, 
$c$ must satisfy one of (1) - (3).

Finally, if $c$ has three consecutive turns then $c$ is not a rib nor a curve in $\calC_{2,0}$. One of 
the two bones between the turns of $c$ must meets a length edge and a width edge, contradicting (3).
\end{proof}

\begin{definition}\label{def: good}
A curve $c\in \calC$ is \emph{good} if it is in $\calC_{2,0}$ or one of the following forms:

\begin{figure}[H]
    \includegraphics[width = 0.7\textwidth]{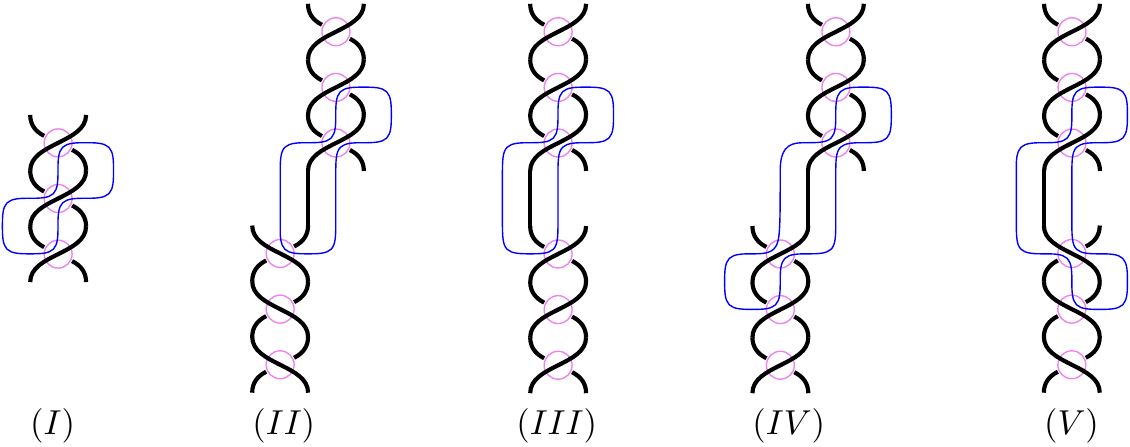}
    \caption{The five configurations of good curves which are not in $\calC_{2,0}$}
    \label{fig: good curves}
\end{figure}
We will say that $c$ is  good of type (I) - (V) accordingly. Otherwise, $c$ is 
called \emph{bad}.
\end{definition}

\begin{remark}\label{rem: Good curevs are good}
Note that if $S$ is a boundary parallel torus, then its intersection curves are good. 
A key observation is that the converse holds. That is, if all the curves of intersection 
of $S$ with $P$ are good then $S$ is a boundary parallel torus. Therefore, our goal in the 
next claims is to show that the curves of intersection of a taut surface $S$ with $\chi(S)=0$ 
are good.
\end{remark}

\begin{lemma}\label{lem: returning is good}
If a curve $c\in\calC$ passes through a bubble 
more than once, then $c$ is good.
\end{lemma}

\begin{proof}
Let $c$ meet the bubble $B_1$ more than once. By Lemma \ref{lem: properties of curves}(5) 
it can do so only in two opposite \I-joints, $\alpha,\beta$. Therefore, at least one of those 
\I-joints, say $\alpha$, is part of a wiggle  $\alpha'$ of $c$. Hence, $\alpha'$ meets 
an adjacent bubble $B_2$. Note that $c \ssm (\alpha'\cup\beta)$ consists of two arcs 
connecting $\alpha'$ and $\beta$ as depicted in Figure \ref{fig:meets bubble twice}. Let $\gamma_R$, 
$\gamma_L$ be the dotted subarcs of $c$ on the right and left of the figure respectively. The 
argument  is divided into cases according to Table \ref{tab: classification table}:

\begin{figure}[H]
    \centering
    \includegraphics{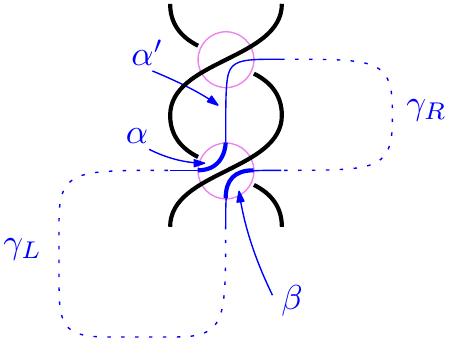}
    \caption{A curve $c$ that meets a bubble twice.}
    \label{fig:meets bubble twice}
\end{figure}

\vskip5pt

\begin{enumerate}
    \item The curve $c$ contains three \I-joints, hence it is not in $C_{2,0}$, $C_{2,2}$ 
    or $C_{0,4}$.
    
    \vskip5pt
    
    \item If $c\in\calC_{4,0}$, then the subarc $\gamma_R$ of $c$ has no joints while 
    $\gamma_L$ has one \I-joint. Hence $c$ is good of type (I), (II) or (III).
    
    \vskip5pt
        
    \item Since the three \I-joints of $c$ in $T$ are not part of the same limb in 
    $\calK_{3,0}$ nor $\calK_{4,0}$, then $c$ cannot be in $\calK_{3,0}+0,2$  or 
    in $\calK_{4,0}+0,2$ (See Table \ref{tab: classification table}).
    
    \vskip5pt
    
    \item The curve $c$ cannot be in $\calK_{4,0}+\calK_{4,0}$: Otherwise $\beta$ 
    is part of a wiggle $\beta'$ of $c$. The wiggle $\alpha'$ (resp. $\beta'$) is part 
    of a limb $\alpha'' \in\calK_{4,0}$ (resp. $\beta'' \in \calK_{4,0})$. Since each limb 
    of $\calK_{4,0}$ has a core, the bubble $B_2$ must be extremal. Then, the subarc $\gamma_R$ 
    of $c$ contains the other wiggle of $\alpha''$. The closed curve 
    which is the union of $\gamma_R$ and an arc on the boundary of the twist region intersects 
    the link diagram twice, and both sub-diagrams bounded by it are non-trivial. This contradicts 
    the assumption that the  diagram is prime.
    
     \vskip5pt
     
     \item A similar argument shows that $c$ cannot be in $\calK_{2,1}$+$\calK_{2,1}$. 
     
     \vskip5pt
     
     \item The curve $c$ cannot be in $\calK_{2,1} + {1,1}$: Otherwise, $\alpha'$ 
     is the wiggle of some $\alpha''\in \calK_{2,1}$ and $\beta$ is a turn. Beside 
     $\alpha'$ and $\beta$,  $c$ has a \D-bone on the subarc $\gamma_L$, 
     and no joints on $\gamma_R$.  Since $\alpha'' \in \calK_{2,1}$ it abuts some $c_0\in\calC_{1,2}$. 
     Hence, $c\cap L$ and $c_0 \cap L$ share endpoints. It follows that the union 
     $(c\cap L) \cup (c_0 \cap L)$ is a component of $L$ passing over 
    at most two wiggles of the diagram (at $c\cap L$) and under at most two wiggles 
     (at $c_0 \cap L$). This contradicts the assumption that $L$ is 3-highly twisted.
     
     \vskip5pt
     
     \item If the curve $c$ is in $\calK_{4,0}+{2,0}$ then $c$ is good of type (V):
     If $\beta$ is part of a wiggle $\beta'$ of $c$, then at most one of $\alpha'$ and 
     $\beta'$ is part of limb in $\calK_{4,0}$. Without loss of generality, assume $\alpha'$ 
     is a wiggle of a limb $\alpha''\in\calK_{4,0}$. Then, $B_2$ is extremal, and the subarc 
     $\gamma_R$ contains the other wiggle of $\alpha''$. The curve which is the union of 
     $\gamma_R$ and an arc on the boundary of the twist region intersects the link diagram 
     twice, in contradiction to the the assumption that the  diagram is prime.  If $\beta$ 
     is a turn, then $B_1$ is extremal, and the wiggle $\alpha'$ is part of a limb 
     $\alpha''\in\calK_{4,0}$. This limb abuts a curve $c_0\in\calC_{2,0}$, and it follows 
     that $c$ is good of type (V).
     
     \vskip5pt
     
     \item Finally, if the curve $c$ is in $\calK_{3,0}+\calK_{3,0}$ then $c$ 
     is good of type (IV):  The wiggle $\alpha'$ is a wiggle of some limb 
     $\alpha''\in\calK_{3,0}$. If $\beta$ is part of a wiggle $\beta'$ of $c$, then 
     $\beta'$ is a wiggle of some other limb $\beta''\in \calK_{3,0}$. It follows that 
     each of $\gamma_R,\gamma_L$ contains exactly one \I-joint, which is impossible.
     If $\beta$ is a turn, then it is the turn of some limb $\beta''$ in $\calK_{3,0}$.
     As the core of $\beta''$ meets the width of the twist region, its wiggle must be 
     on $\gamma_L$. Similarly, the turn of $\alpha''$ must be on $\gamma_L$ as well. 
     This implies that $\gamma_R$ does not contain any joints. If the turn of $\alpha''$ 
     and the wiggle of $\beta''$ are in two different twist regions then the curve 
     abutting (both of) their cores contains three turns. However, this curve is a 
     non-terminal layer curve (in the sense of Lemma \ref{lem: onion}), and those 
     contain at most two turns. Thus, the turn of $\alpha''$ and the wiggle of 
     $\beta''$ are in the same twist region $T'$.  Each of $\alpha''$ and $\beta''$ 
     meets an extremal bubble of $T'$. Then only option for them to close up is if 
     they meet the same extremal bubble of $T'$. It follows that $c$ is good of type (IV).
    
\end{enumerate}
\end{proof}

\begin{proposition}\label{prop: all curves are good or C04}
All curves in $\calC$ are good or in $\calC_{0,4}$.
\end{proposition}

In the proof of the proposition, we will assume in contradiction that such a curve exists. 
The proof will follow from the next four lemmas.

\begin{lemma}\label{lem: turns are external}
Assume that there are bad curves which are not in $\calC_{0,4}$. Let $c$ be an 
innermost bad curve in $P^+$ which is not in $\calC_{0,4}$. Let $D$ be the 
disk bounded by $c$. Then the curve $c$ does not turn or wiggle through a twist 
region $T$ which has a bubble contained  in $D$.
\end{lemma}

\begin{proof}
\underline{Assume $c$ turns at $T$:} Let $B$ be the extreme bubble in a twist region $T$
through which $c$ turns. Since the diagram is 3-highly twisted $T$ contains at least two 
more bubbles let $B'$ be the bubble adjacent to $B$ in $T$. By assumption $B'$ is contained 
in the disk $D$. Consider the curve $c'$ whose \I-joint is opposite to the \I-joint of $c$ 
in $B$. Since $c$ is innermost, the curve $c'$ is good. It must be of good type (I) as in the 
next configuration.

\begin{figure}[H]
\centering
    \includegraphics[width = 0.25\textwidth]{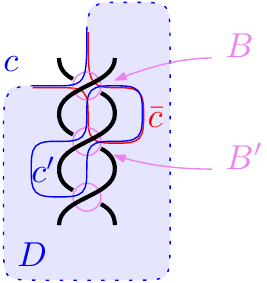}
\end{figure}

The curve $\bar c$ abutting both $c$ and $c'$, passes through the bubble $B'$ twice. 
Hence, by Lemma \ref{lem: returning is good}, $\bar c$ is good.  By the definitions,
it cannot be in $\calC_{2,0}$ nor good of type (I). If $\bar c$ is good of type (II) 
or (III) it has a turn at a bubble $B''$, then $c$ passes twice through $B''$ which 
by Lemma \ref{lem: returning is good} contradicts the assumption that $c$ is bad. 
If $\bar c$ is good of type (IV) or (V) then it meets an extremal bubble $B''$. 
The curve $c$ turns at $B''$ and hence belongs to $C_{2,0}$ which again contradicts 
the assumption that $c$ is bad. 

\vskip 7pt

\underline{Assume $c$ wiggles through $T$:} The curve $c$ wiggles through the bubbles 
$B_1,B_2$. Let $B_0$ be a bubble of $T \cap D$ so that $B_0,B_1,B_2$ are consecutive, 
as in the following figure. 

\begin{figure}[H]
\centering
    \includegraphics[width = 0.27\textwidth]{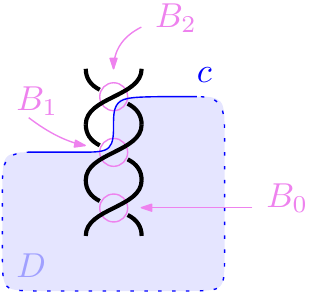}
\end{figure}
Consider the curve $c'$ whose \I-joint is opposite to the \I-joint of $c$ in $B_1$.
The curve $c'$ is contained in
$D$ and wiggles through $T$ passing through the bubbles $B_0,B_1$. 

\begin{figure}[H]
\centering
    \includegraphics[width = 0.23\textwidth]{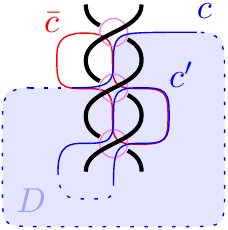}
\end{figure}

By assumption $c'$ must be good, and so it  wiggles through $T$ and then returns to $T$, 
passing through $B_0,B_1,B_0$ in that order. By Lemma \ref{lem: properties of curves}(5), 
the two \I-joints of $c'$ in $B_0$ are opposite sides of the same saddle. Next, consider 
the curve $\bar c$ abutting $c'$ along the two \I-bones of $c'$ connecting $B_0$ and $B_1$. 
The curve $\bar c$ passes through $B_1$ twice. Hence, it is good by Lemma \ref{lem: returning is good}.
It must be of type (I) and in addition passes through $B_2$. It follows that $c$ abuts 
$\bar c$ along the two \I-bones of $\bar c$ connecting $B_1$ and $B_2$. Hence, 
$c$ passes through $B_2$ twice, and by  Lemma \ref{lem: returning is good}
$c$ is good, contradicting the assumption.
\end{proof}

\begin{lemma}\label{lem: innermost does not cross}
Assume that there are bad curves which are not in $\calC_{0,4}$. 
Let $c$ be an innermost bad curve in $P^+$ which is not in $\calC_{0,4}$. Let $D$ 
be the disk bounded by $c$. Then the curve $c$ does not wiggles through a twist region.
\end{lemma}

\begin{proof}
Assume that $c$ wiggles through a twist region $T$. By Lemma \ref{lem: turns are external}, 
the disk $D$ does not contain a bubble of $T$. The curve $c$ wiggles extremely through the 
twist region by passing through two bubbles $B_0,B_1$, where $B_0$ is extremal. By 
Lemma \ref{lem: returning is good}, $c$ meets the bubble $B_0$ once. 

The curve $c'$ turning at $B_0$ is good by choice of $c$. Therefore, 
$c'\in\calC_{4,0}$ is good of type (II) or (III) or $c'\in\calC_{2,0}$ of 
type (i) or (ii) (as in Lemma \ref{lem: possible positive curves}):

\begin{figure}[H]
    \includegraphics[width = 0.7 \textwidth]{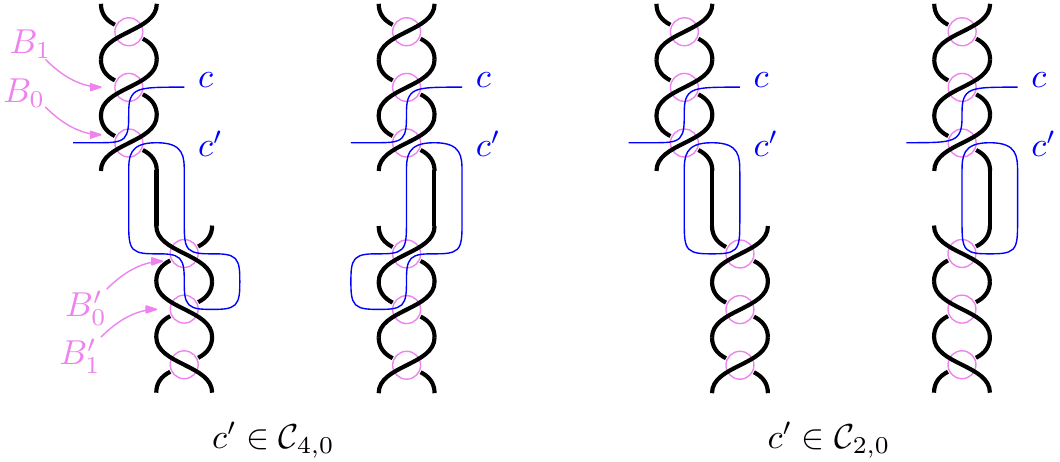}
    \caption{The four configurations of $c'$}\label{fig: configurations of c'}
\end{figure}

Let $T'\ne T$ be the other twist region which $c'$ meets, let $B_0'$ denote 
the extremal bubble in $T'$ through which $c'$ passes, and let $B_1'$ be its 
adjacent bubble.

\vskip7pt

\underline{Case 1.} If $c' \in \calC_{4,0}$, (i.e., as depicted in the left two 
sub-figures in Figure \ref{fig: configurations of c'}), then consider the curve 
$\bar c$ abutting $c$ and $c'$ (shown in orange below). None of the bubbles of 
$\bar c$ belongs to an arc of $\calK$. Therefore, $\bar c \in \calC_{4,0}$, and 
 it follows that it must close up as shown in the dotted curves below:

\begin{figure}[H]
    \includegraphics[]{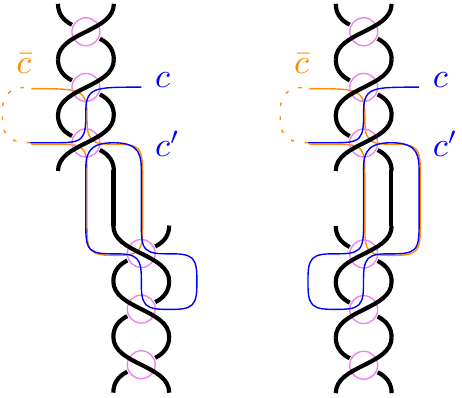}
\end{figure}

As $c$ must follow the dotted bone of $\bar{c}$, we see that $c$ passes through 
the bubbles $B_1$ twice. By Lemma \ref{lem: returning is good}. This contradicts 
 the assumption that $c$ is bad.

\vskip5pt

\underline{Case 2.} Let  $c'\in\calC_{2,0}$, (i.e., as depicted in the right two sub-figures 
of Figure \ref{fig: configurations of c'}) and assume that $c$ \emph{does not} meet $B_0'$. 
Consider, now,  the curve $c''$ whose \I-joint in $B_0'$ is opposite to that of $c'$. 
By the choice  of $c$, the curve $c''$ is good. It must be good of type (I). And the 
configuration is as depicted here:

\begin{figure}[H]
    \includegraphics[]{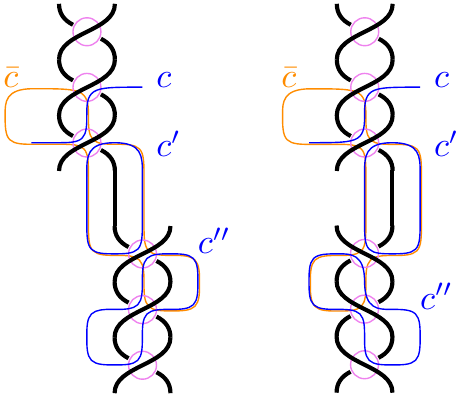}
\end{figure}

Considering the curve $\bar c$ abutting all of $c,c',c''$, we see that it must 
be good by Lemma \ref{lem: returning is good} of type (IV) and (V) respectively. 
As in Case 1, it follows that $c$ meets $B_1$ twice contradicting the assumption that $c$ is bad.

\vskip5pt

\underline{Case 3.} Assume that  $c'\in\calC_{2,0}$ of type (i) and $c$ \emph{does meet} 
$B_0'$ (as in the third, counted from the left, sub-figure of  Figure \ref{fig: configurations of c'}). 
In this situation are three sub-cases to consider: 

\begin{enumerate}
    \item $c$ wiggles through $T'$ at the bubbles $B_0'$ and $B_1'$.
    
 \vskip5pt  
 
 \item $c$ turns at $B_0'$.
 
  \vskip5pt  
 
 \item $c\cap L$ meets the bubble $B_0'$.
    
\end{enumerate}
\vskip7pt
\underline{Sub-case 3.1.}  The curve $c$ wiggles through $T'$ and through the bubbles $B_0'$ 
and $B_1'$.  After passing through $B_1$, the curve $c$ must exit $T$ at its right length 
edge and enter $T'$ on its right length edge before meeting $B_0'$. Otherwise the curve $c$ 
would be forced to pass through a bubble twice in contradiction to Lemma \ref{lem: returning is good}. 
Thus, the curve $c$ is as in the following figure:

\begin{figure}[H]
    \includegraphics{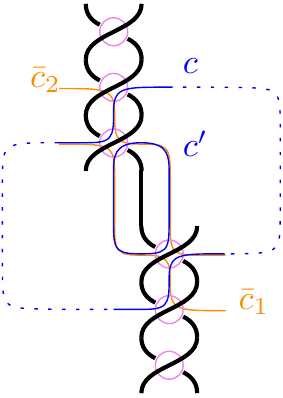}
\end{figure}

If $c\in\calC_{4,0}$ then we get a twist-reduction subdiagram, in contradiction to the assumption. 
Thus, by Lemma \ref{lem: classification of chi'=0}, at least one  of the wiggles of $c$, via 
$B_0,B_1$ or via $B_0',B_1'$,  must be part of an arc $\kappa\in \calK$. Since the curves
$\bar{c}_1,\bar{c}_2$ abutting $c'$ are not in $C_{2,0}$ nor in $\calC_{1,2}$ it is clear that 
$\kappa\notin \calK_{4,0}\cup\calK_{2,1}$. Hence $\kappa\in\calK_{3,0}$. Since $c$ has two wiggles, 
it must be in $\calK_{3,0} + \calK_{3,0}$ (as in Table \ref{tab: classification table}), and 
each of the dotted subarcs (in the figure above) must contain a turn. 

Let $\kappa\in\calK_{3,0}$ be the limb that wiggles through $B_0',B_1'$. Then, $c$ is the terminal 
curve of some vertebra $\tau$ in a curve in $\calC_{2,0}$ (in the sense of Lemma \ref{lem: onion}). 
By retracing backwards, we see that the sequence of curves terminating in $c$ starts with 
$c'\in\calC_{2,0}$, then produces $\bar c_2\in\calC_{4,0}$, and then finally  produce the curve $c$ 
(and the limb $\kappa$). In particular, $\bar c_2$ is a rib, and the \I-bone connecting its 
two turns is a vertebra. The following figure shows an example of such a configuration of curves. 
As one can see, there are subarcs of $\bar c_2$ and $c$ that together bound a twist-reduction 
subdiagram which is a contradiction. 

\begin{figure}[H]
    \includegraphics[]{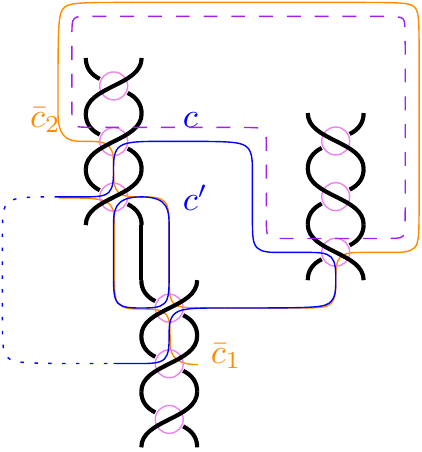}
\end{figure}

\vskip7pt

\underline{Sub-case 3.2} Assume $c$ turns at the bubble $B_0'$. As explained at the beginning 
of the previous sub-case 3.1, the curve $c$ must be as in the following figure:

\begin{figure}[H]
    \centering
    \includegraphics{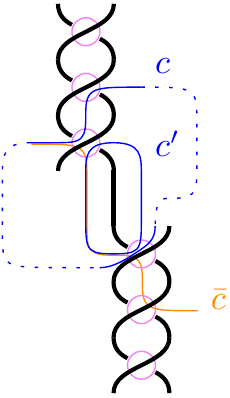}
    \caption{The curve $c'$ is of type (i) and $c$ turns at $B_0'$.}
    \label{fig: subcase 3.2}
\end{figure}

Let $\bar c$ be the curve abutting $c$' as in the figure above. The argument proceeds by 
dividing into cases according to Table \ref{tab: classification table}:

\vskip5pt

\begin{enumerate}
    \item The curve $c$ contains three \I-joints, hence it is not in $C_{2,0}$, $C_{2,2}$ 
    or $C_{0,4}$.
    
    \vskip5pt
    
    \item The curve $c$ is \emph{not} in $\calC_{4,0}$: Otherwise, $c$ has to be a rib. The bone 
    $\tau$ connecting the two turns cannot meet both width edges of the corresponding twist 
    regions by Claim \ref{claim: special triple cases}. Hence it is a vertebra, e.g., as 
    in the following figure:
    
    \begin{figure}[H]
        \centering
        \includegraphics{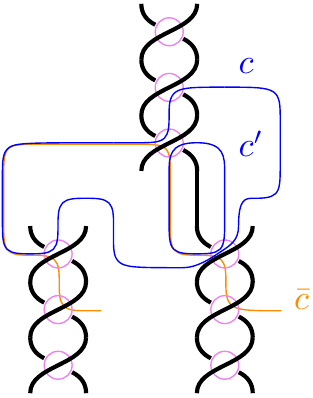}
    \end{figure}
    
    As $\bar c$ has at least five \I-joints, by Table \ref{tab: classification table} 
    $\bar c$ is in $\calK_{3,0}+\calK_{3,0}$,
    and its two cores are those of the limbs in $\calK_{3,0}$.
    The core of $c$ ending in $B_0$ meets the length edge of both twist regions it connects. 
    This is impossible for a core of a limb in $\calK_{3,0}$.
    
    \vskip5pt
        
    \item The curve $c$ cannot be in $\calK_{4,0}+0,2$, $\calK_{2,1}$+$\calK_{2,1}$, 
    or $\calK_{4,0}+\calK_{4,0}$ (See Table \ref{tab: classification table}) since $c$ 
    contains a turn. 
    
    \vskip5pt
    \item The curve $c$ cannot be in $\calK_{3,0}+0,2$: Otherwise, the three \I-joints of $c$ 
    are part of the same limb in $\calK_{3,0}$. The core of such a limb connects a width edge to 
    a length edge. This is not the case here.
    
    \vskip5pt
    
     \item The curve $c$ cannot be in $\calK_{2,1} + {1,1}$ or in $\calK_{4,0}+{2,0}$: 
     Otherwise, the wiggle of $c$ is part of a limb in $\calK_{2,1}$ or $\calK_{4,0}$ 
     respectively. It follows that the curve $\bar{c}$ abutting $c$ and $c'$ in  
     Figure \ref{fig: subcase 3.2} must be in $\calC_{1,2}$ or $\calC_{2,0}$ respectively, 
     which is clearly not the case.
     
     \vskip5pt
     
     \item Finally, the curve $c$ cannot be in $\calK_{3,0}+\calK_{3,0}$: Otherwise, it 
     follows that $\gamma_R$ must contain a single wiggle. However, in this case one can 
     close $\gamma_R$ with an arc along $c'$ to obtain a curve in $P$ intersecting the 
     diagram in two points contradicting the assumption that the diagram is prime.
    
\end{enumerate}

\vskip7pt

\underline{Sub-case 3.3} If $c\cap L$ meets the bubble $B_0'$, then it is as depicted below:

\begin{figure}[H]
    \includegraphics[]{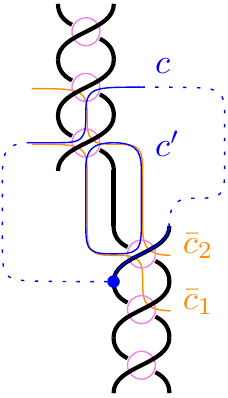}
\end{figure}

If $c\in\calC_{2,2}$ then the left dotted line passes through no bubbles or intersection points,  
and we get a contradiction to the parity. Therefore, by Lemma \ref{lem: classification of chi'=0}, 
$c$ must contain some arc $\kappa\in \calK$. Clearly, $\kappa$ must contain the 
subarc of $c$ wiggling through $T$ via $B_0,B_1$. 
The curve $\bar c_1$ is not in $\calC_{1,2}\cup \calC_{2,0}$ and therefore 
$\kappa\notin \calK_{2,1}\cup \calK_{4,0}$. If follows $\kappa\in \calK_{3,0}$ and is 
the result of the process $c'$, then $\bar c_1$, then $\kappa$, which is discussed in the
proof of Lemma \ref{lem: onion}.  If this is the case then, as in the end of sub-case 3.1, 
a subarc of $\bar c_1$ and a subarc of $c$ bound a twist-reduction subdiagram as in the 
example shown in the following figure:

\begin{figure}[H]
    \includegraphics[width = 0.2 \textwidth]{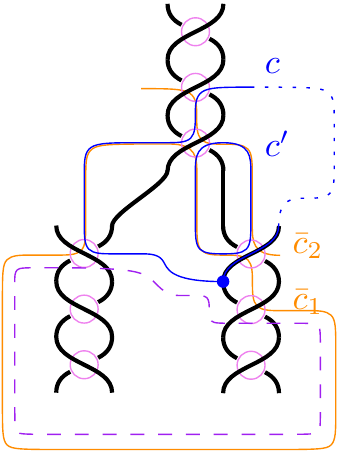}
\end{figure}

\vskip7pt

\underline{Case 4.} Assume that  $c'\in\calC_{2,0}$ of type (ii) 
 and $c$ \emph{does meet} $B_0'$. (as depicted in the rightmost sub-figure of 
 Figure \ref{fig: configurations of c'}). Let $\bar c_1$ be the curve abutting 
 $c$ as in the following figure.

\begin{figure}[H]
    \centering
    \includegraphics{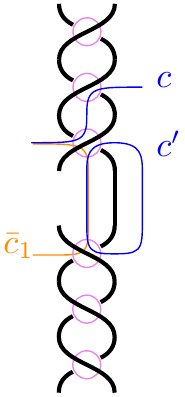}
\end{figure}

The wiggle of $c$ passing via $B_0,B_1$ is not part of any arc $\kappa$ in $\calK$: 
If it were, then the curve $\bar c_1$ will either be in $\calC_{2,0}\cup \calC_{1,2}$ 
or will be the non-terminal layer curve of a process terminating in $\kappa$. Clearly, 
$\bar c_1$ is not in $\calC_{2,0} \cup \calC_{1,2}$. It is also not a non-terminal 
layer of a process defining $\calK_{3,0}$ as in Lemma \ref{lem: onion}, 
since at any step of a process every curve is a rib, i.e., it has two turns and a wiggle, 
and the next step of the process abuts the \I-bone connecting its the two turns, 
however here $c$ does not abut the \I-bone of $\bar c_1$ connecting its two turns.

As in the previous sub-case there are three further sub-sub-cases to consider: 
\begin{enumerate}
    \item $c$ wiggles through $T'$ at the bubbles $B_0'$ and $B_1'$.
    
 \vskip5pt  
 
 \item $c$ turns at $B_0'$.
 
  \vskip5pt  
 
 \item $c\cap L$ meets the bubble $B_0'$.
    
\end{enumerate}

\vskip7pt

\underline{Sub-case 4.1} If $c$ wiggles through the bubbles $B_0',B_1'$, the 
exact same argument as above shows  that the subarc of $c$ passing through 
$B_0',B_1'$ is not a subarc of any arc in $\calK$.  It follows that 
$c\in\calC_{4,0}$ and bounds a twist-reduction subdiagram which is a contradiction.

\begin{figure}[H]
    \includegraphics{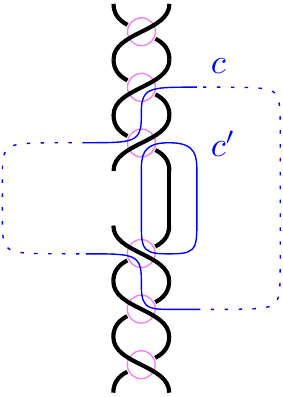}
\end{figure}

\vskip7pt

\underline{Sub-case 4.2} If $c$ turns at the bubble $B_0'$. 
The curve $c$ must be as in the following figure:
\begin{figure}[H]
    \centering
    \includegraphics{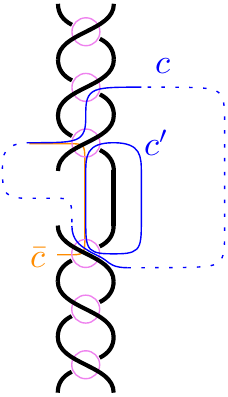}
    \caption{The curve $c'$ is of type (ii) and $c$ turns at $B_0'$}
    \label{fig: case 4.2}
\end{figure}
The argument is further divided into cases according to Table \ref{tab: classification table}:

\vskip5pt

\begin{enumerate}
    \item The curve $c$ contains three \I-joints, hence it is not in $C_{2,0}$, $C_{2,2}$ 
    or $C_{0,4}$.
    
    \vskip5pt
    
    \item The curve $c$ is not in $\calC_{4,0}$: Otherwise, $c$ has to be a rib. The 
    bone $\tau$ connecting the two turns cannot meet both width edges of the corresponding 
    twist regions by Claim \ref{claim: special triple cases}. Hence it is a vertebra, as in 
    the following figure. 
    
    \begin{figure}[H]
        \centering
        \includegraphics{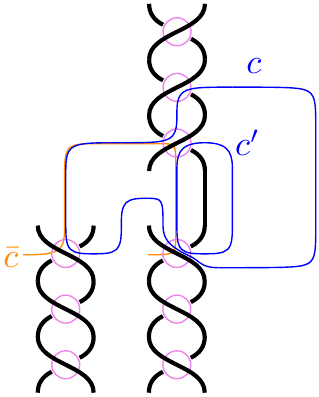}
    \end{figure}
    
    It follows that the curve $\bar c$ abutting $c$ and $c'$ has at least three consecutive 
    turns in contradiction to Lemma \ref{lem: two turns}.
    
    \vskip5pt
        
    \item The curve $c$ cannot be in $\calK_{4,0}+0,2$, $\calK_{2,1}$+$\calK_{2,1}$ 
    or $\calK_{4,0}+\calK_{4,0}$ (See Table \ref{tab: classification table}) since $c$ 
    contains a turn. 
    
    \vskip5pt
    \item The curve $c$ cannot be in $\calK_{3,0}+0,2$: Otherwise the three \I-joints of $c$ 
    are part of the same limb in $\calK_{3,0}$. However this would imply that the three \I-joints 
    of $c$ are consecutive, which is impossible by the checkerboard coloring of the diagram.
    
    \vskip5pt
    
     \item The curve $c$ cannot be in $\calK_{2,1} + {1,1}$ or in $\calK_{4,0}+{2,0}$: 
     Otherwise, the wiggle of $c$ is part of a limb in $\calK_{2,1}$ or $\calK_{4,0}$ 
     respectively. It follows that the curve $\bar{c}$ abutting $c$ and $c'$ in  
     Figure \ref{fig: case 4.2} must be in $\calC_{1,2}$ or $\calC_{2,0}$ respectively, 
     which is clearly not the case.
     
     \vskip5pt
     
     \item Finally, the curve $c$ cannot be in $\calK_{3,0}+\calK_{3,0}$: Otherwise, 
     it follows that the curve $\bar c$ abutting $c$ has three consecutive turns. This 
     is impossible by Lemma \ref{lem: two turns}.
    
\end{enumerate}

\vskip7pt

\underline{Sub-case 4.3} If $c\cap L$ meets the bubble $B_0'$, then by 
Table \ref{tab: classification table} $c$ is either in $\calC_{2,2}$, 
or $\calK_{2,1}+1,1$ or $\calK_{2,1}+\calK_{2,1}$. It cannot be in
$\calK_{2,1}+1,1$ or $\calK_{2,1}+\calK_{2,1}$ since otherwise the 
dotted subarc $\gamma_R$ has a turn or a wiggle respectively, which would 
contradict primeness. Thus, $c\in\calC_{2,2}$.

\begin{figure}[H]
    \includegraphics{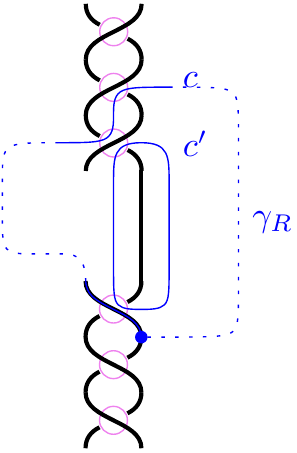}
\end{figure}

If $c\cap L$ passes over one crossing of $L$, then $c$ bounds a twist-reduction 
subdiagram as in the following figure:

\begin{figure}[H]
    \includegraphics{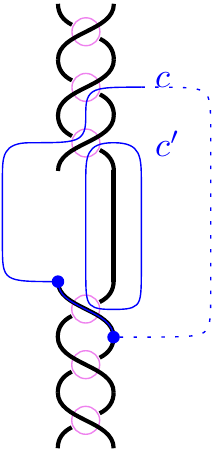}
\end{figure}

Otherwise, we are in the following configuration:

\begin{figure}[H]
    \includegraphics{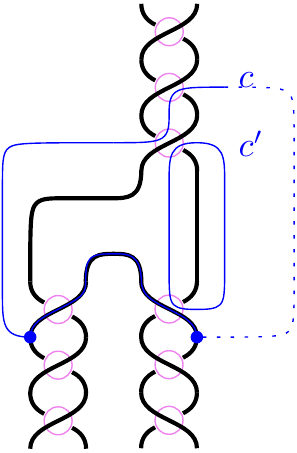}
\end{figure}

Now consider how the curve $\bar c_1$ abutting $c'$ on its left can close up. 
It must be as depicted here:

\begin{figure}[H]
    \includegraphics{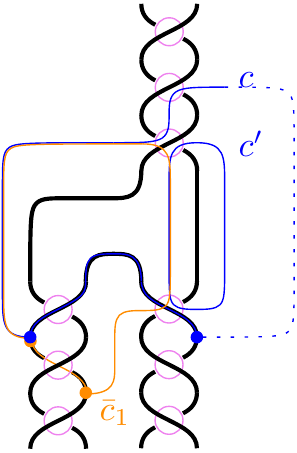}
\end{figure}

Thus, $\bar c_1$ must be as one of the configurations that were ruled out in 
Claim \ref{claim: special triple cases}.
\end{proof}

\begin{lemma}
Assume  that there are bad curves which are not in $\calC_{0,4}$. 
Let $c$ be an innermost bad curve in $P^+$ which is not in $\calC_{0,4}$. Let $D$ 
be the disk bounded by $c$. Then $c\notin \calC_{4,0}$.
\end{lemma}

\begin{proof}
By lemma \ref{lem: innermost does not cross}, $c$ only turns. However, this is 
impossible by Lemma \ref{lem: two turns}.
\end{proof}

\begin{lemma}
Assume that there are bad curves which are not in $\calC_{0,4}$. Let $c$ be an 
innermost bad curve in $P^+$ which is not in $\calC_{0,4}$, and let $D$ be the 
disk bounded by $c$. Then $c\notin \calC_{2,2}$.
\end{lemma}

\begin{proof}
Assume $c\in\calC_{2,2}$, Since $c \ssm (c\cap L)$ passes through two bubbles, 
the endpoints of (a small continuation of) $c\cap L$ have the same color in the 
checkerboard coloring of $P\ssm D(L)$. Thus it is one of the following:

\begin{figure}[H]
    \includegraphics{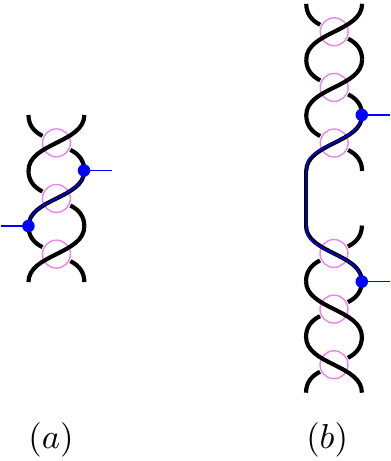}
\end{figure}

The complement $c \ssm (c\cap L)$ contains two turns, in twist regions which do not 
contain bubbles in $D$. Thus, the possible configurations are:

\begin{figure}[H]
    \includegraphics[width = 0.9\textwidth]{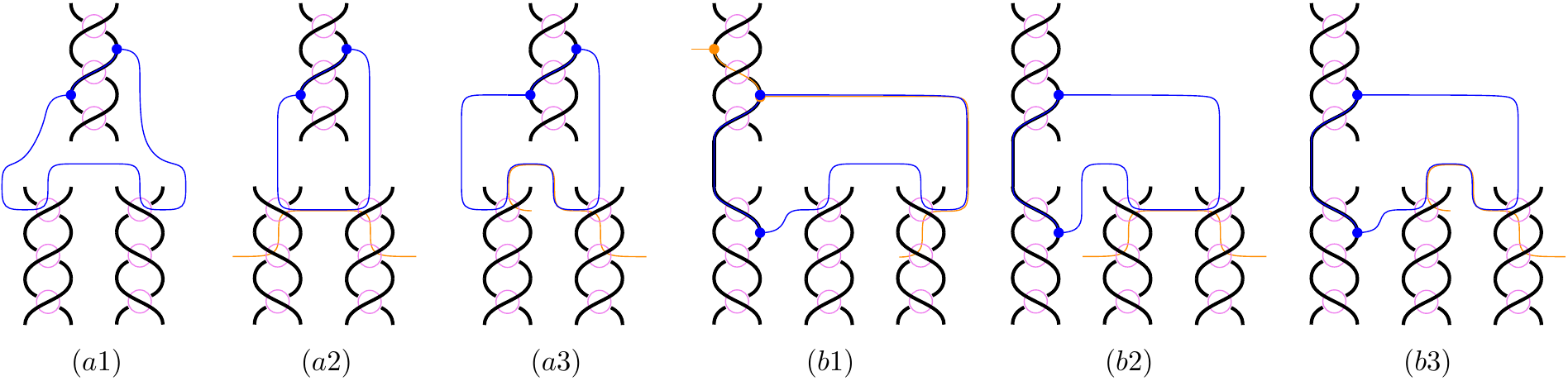}
\end{figure}

In each of the cases, consider the curve $\bar c$ abutting $c$ (a subarc of 
which is shown in orange). 

Case $(a1)$ is ruled out by Claim \ref{claim: special triple cases}.
In cases $(a2),(b1),(b2)$, it is clear that the curve 
$\bar c \in \calC_{4,0}\cup \calC_{2,2}$ and bounds a twist-reduction 
subdiagram, which is a contradiction. Cases $(a3),(b3)$ are impossible by 
Lemma \ref{lem: two turns}. 
\end{proof}

\begin{proof}[Proof of Proposition \ref{prop: all curves are good or C04}]
Assume in contradiction that there are bad curves which are not in $\calC_{0,4}$. 
Let $c$ be an innermost bad curve in $P^+$ which is not in $\calC_{0,4}$.

By the last two lemmas, $c$ is not in $\calC_{4,0}$ nor in $\calC_{2,2}$.  Thus, 
by Lemma \ref{lem: classification of chi'=0}, $c$ must contain an arc of $\calK$. 
Since every arc in $\calK$ wiggles through some twist region, we get a
contradiction  to Lemma \ref{lem: innermost does not cross}.
This contradiction finishes the proof of 
Proposition \ref{prop: all curves are good or C04}.
\end{proof}

\begin{corollary}\label{cor: atoroidal}
If the diagram of $L$ is 3-highly twisted, connected, prime, and twist-reduced then 
$\bbS^3 \ssm \calN(L)$ is atoroidal. In particular, $L$ is prime.
\end{corollary}

\begin{proof}
Let $S\subset \bbS^3 \ssm \calN(L)$ be an incompressible taut torus. Let $\calC$ be 
the curves of intersection of $S$ with $P^\pm$. Since $S$ has no boundary, 
$\calC_{0,4} = \emptyset$. By Proposition \ref{prop: all curves are good or C04} 
all curves in $\calC$ are good. Since each of the curves $c \in \calC$ bounds a disk in either 
$B^+$ or $B^-$ then after adding these disks along the curves in $\calC$ it is readily seen 
that the torus $S$ is boundary parallel.

If $L$ was a composite knot, then the swallow-follow torus would be an essential torus in 
$\bbS^3 \ssm L$.
\end{proof}

\begin{proposition}\label{prop: no curves in C04}
If the diagram of $L$ is 3-highly twisted, connected, prime and twist-reduced 
then $\bbS^3 \ssm \NN(L)$ is unannular.
\end{proposition}

\begin{proof}
Let $S$ be an essential annulus. Since $L$ is prime by Corollary \ref{cor: atoroidal}, 
the annulus can be assumed not to have meridional boundary components. Thus, we may 
assume that $S$ is taut. Let $\calC$ be its curves of intersection with $P$. By 
Proposition \ref{prop: all curves are good or C04} all curves in $\calC$ are either 
good or in $\calC_{0,4}$. Since $S$ has boundary components, not all curves in $\calC$ 
are good, and there is at least one curve $c\in\calC_{0,4}$. We show that this 
is impossible.

If there exists a curve $c\in\calC_{0,4}$ then the curve $\bar c$ abutting $c$ has two 
intersection points which are not connected by an arc of $\bar{c}\cap L$ therefore 
$\bar c \in \calC_{0,4}$. Repeating this argument shows that all the curves are in 
$\calC_{0,4}$, i.e.,  $\calC = \calC_{0,4}$.

\vskip7pt

\underline{Case 1.} There exists a curve $c\in\calC_{0,4}$ which passes twice at the 
same twist region. 

Denote the two connected components of $c\cap L$ by $\alpha$ and $\beta$. 
Let $n$ be the number of crossings of $T$ in-between $\alpha$ and $\beta$. 
We further divide the proof into sub-cases depending on $n$.

\vskip7pt

\underline{Sub-case 1.0:} $n=0$. By Lemma \ref{lem: properties of curves}(5), 
$\alpha$ and $\beta$ do not meet the same bubble of $T$. Since we assume $n=0$, 
they must meet adjacent bubbles of $T$. The annulus $S$ must spiral between the 
strands of $L\cap T$.  Thus we obtain a disk of Type 2 (as in 
Lemma \ref{lem: normal form}) and hence, by the definition of 
$\calC^+$ and $\calC^-$ as in the beginning of
\S\ref{subsec: Curves of intersection}  this curve does not appear in $\calC$.

\vskip7pt

\underline{Sub-case 1.1:} $n=1$. The tangle $L\cap T$ has two components 
$\lambda_1,\lambda_2$. Let $l_1,l_2$ be the corresponding components of $L$ 
 (possibly $l_1=l_2$). Because $n = 1$ the arcs $\alpha$ and $\beta$ meet the 
 same string of $L\cap T$, say $\lambda_1$. Hence the two boundary components 
of $S$ are contained in the same component $l_1$ of $L$. If $l_1=l_2$, then 
there exists a curve $c'\in\calC_{0,4}$ which meets the bridge in-between 
$\alpha$ and $\beta$, and this curve must be as in Sub-case 1.0. Thus, we may 
assume that $l_1\ne l_2$. Consider the disk $\Delta$ as depicted in
Figure \ref{fig: C04 proof}(a). Its interior intersects $L$ in a single 
point in $l_2$, and its boundary is the union of an arc on $S$ and an arc on 
$l_1$. 

The manifold $\NN(S)\cup \NN(l_1)$ has two torus boundary components 
$U$ and $V$. See Figure \ref{fig: C04 proof}(b). Let $U$ be the torus that 
meets $\Delta$. Let $U_-$ be the component of $\bbS^3\ssm U$ containing $l_2$,
and let $U_+$ be the other component. The torus $U$ is incompressible in $U_-$, 
as such a compression must be on $\Delta$ and $\Delta$ does intersect $l_2$ once. 
It is also incompressible in $U_+$, as if  a compression disk exists then since 
it cannot intersect $l_1$, it gives a compression of the annulus $S$, in 
contradiction to the incompressibility of $S$. By Corollary \ref{cor: atoroidal}, 
$U$ must be boundary parallel to either 
$\partial \NN (l_2)$ or $\partial \NN (l_1)$. 

\begin{figure}[ht]
\centering

\subfigure[]{
   \begin{overpic}[height=3.5cm]{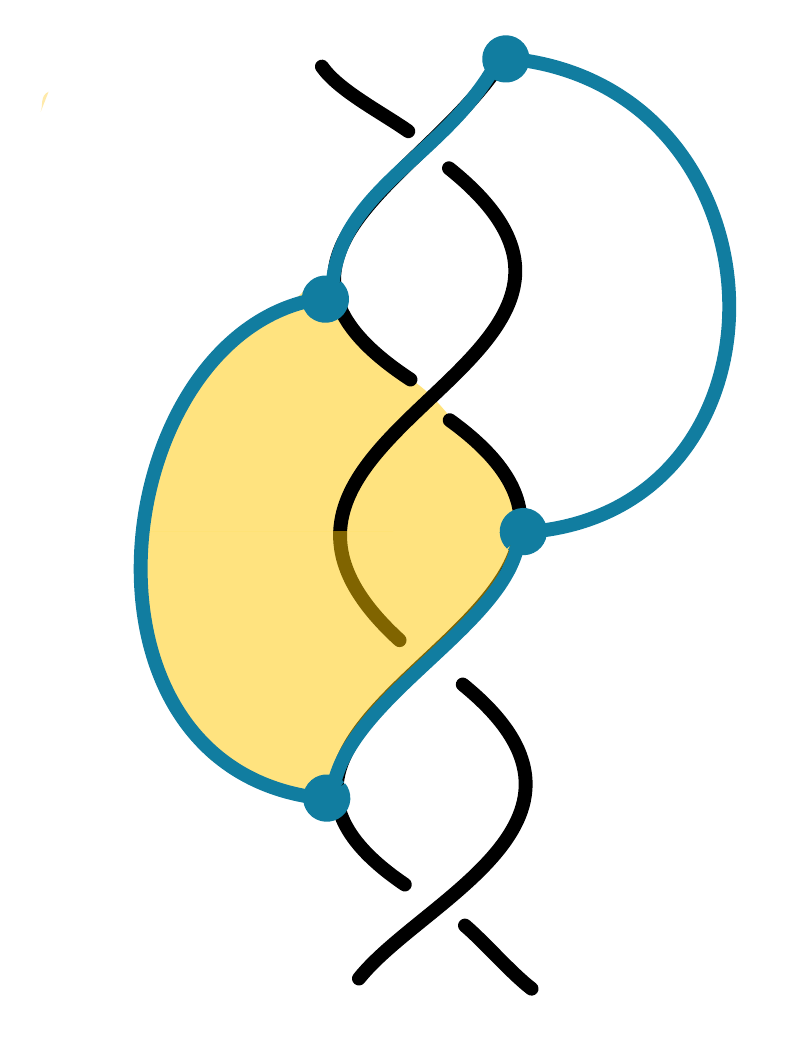}
\put(17,40){$\Delta$}
\put(30,-3){$l_2$}
\put(50,-3){$l_1$}
\put(38,30){\Tiny$\beta$}
\put(38,80){\Tiny$\alpha$}
\end{overpic}
}
\hspace{2cm}
\subfigure[]{
   \begin{overpic}[height=4cm]{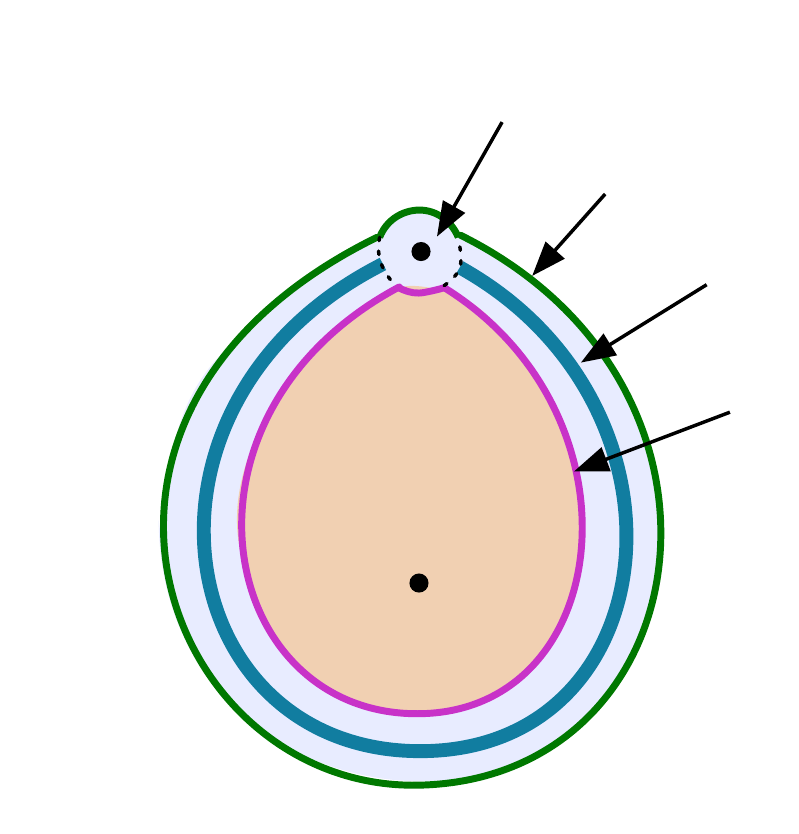}
\put(60,88){$l_1$}
\put(73,75){$V$}
\put(85,65){$S$}
\put(90,50){$U$}
\put(42,30){$l_2$}
\put(45,45){$U^-$}
\end{overpic}
}

\caption{(a) The disk $\Delta$. (b) A cross section of the twist box, 
the annulus $S$ and the tori $U,V$}  
    \label{fig: C04 proof}
\end{figure}
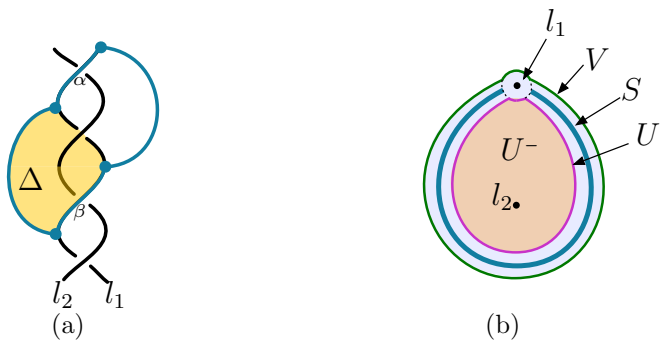

 If $U$ is parallel to $\partial \NN (l_2)$, then since $l_1$ is parallel to 
a curve in $U$ crossing $\Delta$ once there exists an annulus $A\subset \bbS^3 \ssm L$ 
whose boundary is $l_1\cup l_2$. This annulus $A$ is incompressible, since otherwise 
$l_1 \cup l_2$ would be a 2-component unlink that is not linked with $L$, i.e., $L$ 
is split, contradicting Corollary \ref{cor: L is nonsplit}. The annulus $A$ is trivially 
boundary-incompressible because the boundary components of $A$ are on two different 
components of $L$. If we run the argument for $A$ instead of $S$, Case 1.1 cannot 
occur because the boundary components of $A$ are on two different components of $L$.

If $U$ is parallel to $\partial \NN (l_1)$, the intersection $\Delta \cap U$ is a 
curve on $U$ which meets the meridian of $\NN(l_1)$ exactly once: As  if it meets 
it more than once, then the union $\NN(\Delta) \cup \NN(l_1)$ determines a 
once-punctured non-trivial lens space contained in $\bbS^3$, which is impossible. 
Thus, $\partial \Delta$ which is parallel to $\Delta\cap U$ is also parallel to $l_1$. 
Therefore, the arcs $\partial \Delta \ssm l_1 \subset S$ and 
$l_1 \ssm \partial \Delta \subset L$ bound a disk. Since the arc $\partial \Delta \ssm l_1$
connects different components of $S$ it is an essential arc, and the disk is a boundary
compression for $S$,  which is a contradiction.

\vskip7pt

\underline{Sub-case 1.2:} $n\ge 2$.
As the boundary of the annulus $S$ must pass through every other bridge in $T$, there must be 
another curve of $\calC_{0,4}$ in between $\alpha$ and $\beta$. By choosing an innermost such 
curve we are back in one of the previous cases.

\vskip7pt

\underline{Case 2.} None of the curves $c\in\calC_{0,4}$ passes twice at the same twist region.
Let $c_1$ be a curve in $\calC_{0,4}$, then  $c_1\cap L$ is the disjoint union of two arcs 
$\alpha,\beta$. Since all of the curves are in $\calC_{0,4}$ and at least one 
curve passes over a non-extremal bubble, we may assume, by changing $c_1$ that one 
of the components, say $\alpha_1$, passes over a non-extremal bubble in a twist region. 
The component $\beta$ cannot pass over one bubble, as in this case, either $c_1$ 
passes twice in the same twist region, or defines a twist-reduction subdiagram. 
Thus, $\beta$ must be the arc connecting two twist regions, passing over their 
two extremal bubbles, and the situation is as depicted in the following figure.

\begin{figure}[H]
    \includegraphics[width = 0.14\textwidth]{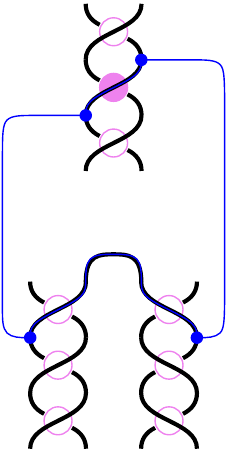}
\end{figure}
This case was ruled out in Claim \ref{claim: special triple cases}.
\end{proof}

\begin{remark}\label{rem: annular atoridal links}
The Sub-case 1.1  $n = 1$ in the proof of Proposition \ref{prop: no curves in C04}
follows also from the following well known general statement:

Non-split, annular atoroidal links in $\bbS^3$ are either torus 
knots or a link consisting of a torus knot on the ``standard torus'' $\bbT^2$ in $\bbS^3$ 
and one or both of the core curves of the solid tori components of $\bbS^3 \ssm \calN(\bbT^2)$.
\end{remark}

For completeness we include a proof.

\begin{proof}
Let $L \subset \bbS^3$  be a non-split and an atoroidal link in $\bbS^3$ containing an 
essential annulus $A$. If $L$ is a knot then boundary $A$ cuts $\partial \NN(L)$
into two annuli $A_1$ and $A_2$. The surfaces $A \cup A_1$ and $A \cup A_2$ are 
tori which bound solid tori $V_1$ and $V_2$ as $L$ is atoroidal. The two solid tori are 
glued to each other along $A$ and since the result together with a regular neighborhood 
of $L$ is $\bbS^3$, then by Seifert (see \cite{Seifert}), their complement is
a regular neighborhood of a torus knot.

If $L$ is a  non-hyperbolic, non-split link whose exterior is atoroidal  then  by
\cite{thurston} it is a Seifert link, i.e. its exterior is a Seifert fiber space.  
Links in Seifert spaces were classified by Burde and Murasugi in \cite{Burde-Murasugi}.
They are either a connected sum of Hopf links or consist of a union of Seifert fibers 
in some Seifert fibration of $\bbS^3$. Atoroidal such links can contain at most three 
fibers. Hence the link $L$ is a torus knot $K$ on an unknotted solid torus $\mathbb{T}$ 
and the  Hopf link which is the core curves of the  complementing solid tori.
\end{proof}

\begin{proof}[Proof of Theorem \ref{thm: highly twisted implies hyperbolic}]
By Thurston \cite{thurston}, it suffices to prove that $\bbS^3 \ssm \NN(L)$ has 
incompressible boundary, and is irreducible, atoroidal and unannular. By 
Corollary \ref{cor: L is nonsplit} it has incompressible boundary and is 
irreducible, by Corollary \ref{cor: atoroidal} it is atoroidal, and by 
Proposition \ref{prop: no curves in C04} it is unannuluar.
\end{proof}

\medskip

\section{Essential holed spheres in highly twisted link complements}\label{sec: essential holed spheres}

In this section we use Theorem \ref{thm: highly twisted implies hyperbolic} to show that 
certain holed spheres in the complement of a highly twisted links are essential. We begin 
with three definitions:

\begin{definition}\label{def: relative diagram}
Let $D(T)$ be a projection of a  tangle $(B, T)$ onto a disk $\Delta \subset P$ 
where $\partial\Delta$ is the simple closed curve $\gamma \subset P$  and  so
that the end points of the strings of $T$ denoted by $\partial T$ are contained 
in $\gamma$. We call $D(T)$ a \emph{tangle diagram}.
Let $\bar D(T)$ denote the reflection of $D(T)$ along $P$ (i.e. the diagram 
with the same projection, but with reverse crossing data).
\end{definition}

\begin{definition}\label{def: relative properties}
The diagram $D(T)$ is \emph{relatively prime} (resp.\emph{relatively twist reduced}, 
\emph{relatively $k$-highly twisted}) if the link diagram obtained by gluing $D(T)$ 
and $\bar D(T)$ along their boundary is  prime (resp. twist reduced,  $k$-highly twisted).
\end{definition}

\begin{definition}
A surface $S$ in $(\bbS^3,\calL)$ is \emph{pairwise-incompressible} if every disk 
$D$ in $\bbS^3$ with $\partial D = D \cap S$, and which intersects $\calL$ transversely 
in a single point, is isotopic to a disk in $S$ by an isotopy preserving $\calL$ set-wise.

The surface $S$ is \emph{acylindrical} if the complement of $S$ in 
$\bbS^3  \ssm \calN(\calL)$ contains no essential annuli whose boundary 
is on $S\cup\partial\calN(\calL)$.
\end{definition}

\begin{theoremA}\label{thm: General vertical spheres} Let $D(L)$ be  $3$-highly 
twisted diagram and let $\gamma$ be a simple closed curve in $D(L)$ intersecting 
$D(L)$ transversely. Assume that both tangle diagrams bounded by $\gamma$ are 
connected,  relatively prime, relatively twist-reduced, 
relatively $k$-highly twisted and contain at least two twist regions each. Let 
$\Sigma$ be the sphere in $\bbS^3$ which intersects $P$ transversely in $\gamma$, 
and does not intersect $\calL$ outside $\gamma$. Then the punctured sphere 
$\Sigma' = \Sigma \ssm \calN(\calL)$ is incompressible, boundary-incompressible, 
pairwise-incompressible and acylindrical in  $\bbS^3 \ssm \calN(\calL)$.
\end{theoremA}

\begin{proof}[Proof of Theorem \ref{thm: General vertical spheres}]

Assume in contradiction that $\Sigma' = \Sigma \ssm \calN(\calL)$ is either incompressible,
boundary-incompressible, pairwise-incompressible or acylindrical in  $\bbS^3 \ssm \calN(\calL)$.

Let $B_1,B_2$ be the two 3-balls that $\Sigma$ bounds in $\bbS^3$, and let 
$E_1=P\cap B_1, \; E_2 = P \cap B_2$ be the corresponding two disks in $P$ 
bounded by $\gamma$. The induced tangle diagrams on $E_1$ and $E_2$ are assumed 
to be relatively prime, relatively twist-reduced and relatively $3$-highly twisted. 
After doubling each of $E_1$ and $E_2$, as in Definition \ref{def: relative properties}, 
we get two link diagrams $D(L_1),D(L_2)$  which are connected, prime, twist-reduced and 
$3$-highly twisted. By Theorem \ref{thm: highly twisted implies hyperbolic}, the associated 
links $\calL_1,\calL_2$ are hyperbolic.

\begin{enumerate}
    \item If $\Sigma'$ is compressible then  the doubling of 
    an innermost  compressing   disk $\Delta \subset B_i$  will give rise to an essential 2-sphere 
   in $\bbS^3 \ssm \calN(\calL_i)$. 

\vskip5pt

    \item If $\Sigma'$ is  boundary compressible with boundary compression disk 
    $\Delta \subset B_i$, then  the doubling of $\Delta$ along $\Sigma'\cap \Delta$ 
    results in a disk which is bounded by a component of $\calL_i$. 
    Hence either $\calL_i$ is the unknot or has a split unknot component.
    
    \vskip5pt
    
    \item If $\Sigma'$ is pairwise  boundary compressible, then the doubling of the 
    essential disk $\Delta \subset B_i$ intersecting $\calL$ once is a 2-sphere 
    intersecting the boundary of $\NN(\calL_i)$ in two meridians.  Thus $\calL_i$ 
    is not prime.
    
    \vskip5pt
    
    \item If  $\Sigma'$ contains an essential annulus $A\subset B_i$ (i.e., $\Sigma'$ 
    is cylindrical), then $\calL_i$ is toroidal if both  boundaries of $A$ are on 
    $\Sigma'$, or annular if one of these boundaries is on $\Sigma'$ and the other on
    $\partial\calN(\calL)$.
\end{enumerate}

In all these cases we get that one of the links $\calL_1,\calL_2$ is not hyperbolic 
and thus a contradiction.
\end{proof}

\bibliographystyle{amsplain}
\bibliography{biblio}

\providecommand{\bysame}{\leavevmode\hbox to3em{\hrulefill}\thinspace}
\providecommand{\MR}{\relax\ifhmode\unskip\space\fi MR }
\providecommand{\MRhref}[2]{%
  \href{http://www.ams.org/mathscinet-getitem?mr=#1}{#2}
}
\providecommand{\href}[2]{#2}
\begin{thebibliography}{10}

\bibitem{Burde-Murasugi}
Gerhard Burde and Kunio Murasugi, \emph{Links and seifert fiber spaces}, Duke
  Mathematical Journal \textbf{37} (1970), no.~1, 89 -- 93.

\bibitem{futerPurcellhyperbolic}
David Futer, Efstratia Kalfagianni, and Jessica~S Purcell, \emph{Hyperbolic
  semi-adequate links}, Communications in Analysis and Geometry \textbf{23}
  (2015), no.~5, 993--1030.

\bibitem{FuterPurcell}
David Futer and Jessica Purcell, \emph{Links with no exceptional surgeries},
  Commentarii Mathematici Helvetici \textbf{82} (2007), 629 -- 664.

\bibitem{giambrone}
Adam Giambrone, \emph{Combinatorics of link diagrams and volume}, Journal of
  Knot Theory and Its Ramifications \textbf{24} (2015), no.~01, 1550001.

\bibitem{Hatcher-Thurston}
Alan Hatcher and William Thurston, \emph{Incompressible surfaces in 2-bridge
  knot complements}, Inventiones \textbf{79} (1985), 225--246.

\bibitem{lackenby2000word}
Marc Lackenby, \emph{Word hyperbolic dehn surgery}, Inventiones mathematicae
  \textbf{140} (2000), no.~2, 243--282.

\bibitem{lazarovich2021highly}
Nir Lazarovich, Yoav Moriah, and Tali Pinsky, \emph{Highly twisted plat
  diagrams}, arXiv preprint arXiv:2111.14231 (2021).

\bibitem{lustig2012large}
Martin Lustig and Yoav Moriah, \emph{Are large distance heegaard splittings
  generic?}, Journal f{\"u}r die reine und angewandte Mathematik (Crelles
  Journal) \textbf{2012} (2012), no.~670, 93--119.

\bibitem{Menasco}
William Menasco, \emph{Closed incompressible surfaces in alternating knot and
  link complements}, Topology \textbf{23} (1984), no.~1, 37--44.

\bibitem{minsky-moriah:surplus}
Yair Minsky and Yoav Moriah, \emph{Discrete primitive-stable representations
  with large rank surplus}, Geometry \& Topology \textbf{17} (2013), no.~4,
  2223--2261.

\bibitem{Seifert}
H.~Seifert, \emph{{Topologie dreidimensionaler gefastner Raume}}, Acta. Math.
  \textbf{60} (1933), 147--238.

\bibitem{thurston}
William~P. Thurston, \emph{The geometry and topology of three-manifolds},
  Princeton Univ. Math. Dept. Notes, 1979.

\end{thebibliography}

\end{document}